\documentclass[a4paper]{article}
\newcommand{\oneandhalfspace}{\renewcommand{\@defaultbaselinestretch}{1.1}}

\usepackage[active]{srcltx}
\usepackage{setspace}
\usepackage[USenglish]{babel} 
\usepackage[T1]{fontenc}
\usepackage[ansinew]{inputenc}
\usepackage{lmodern}
\usepackage{graphicx}
\usepackage{geometry}
\usepackage{amsmath}
\usepackage{amsthm}
\usepackage{amsfonts,amssymb,latexsym}
\usepackage{color}
\usepackage{float}
\usepackage{enumerate}

\newtheorem{thm}{Theorem}[section]

\newtheorem{lemma}[thm]{Lemma}
\newtheorem{prop}[thm]{Proposition}
\newtheorem{defn}[thm]{Definition}
\newtheorem{remark}[thm]{Remark}
\newtheorem{cor}[thm]{Corollary}
\newtheorem{conj}[thm]{Conjecture}

\newtheorem*{notation}{Notation}
\numberwithin{equation}{section}
\numberwithin{figure}{section}

\renewcommand{\atop}[2]{\genfrac{}{}{0pt}{}{#1}{#2}}
\newcommand{\amend}[2]{#2}
\newcommand{\lf}[1]{#1}
\newcommand{\ft}[1]{#1}
\newcommand{\delete}[1]{}

\def\dd{\mathrm{d}}
\def\Efcc{E^{\mathrm{fcc}}}
\def\estar{e^\ast}
\def\Vstar{V^\ast}
\def\meas{\mathrm{meas}}
\def\conv{\mathrm{conv}}
\def\a{\ensuremath{\mathcal{A}}}
\def\angset{\ensuremath{A}}

\def\b{\ensuremath{\mathcal{B}}}

\def\c{\ensuremath{\mathcal{L}}}

\def\cfcc{\ensuremath{\mathcal{L}_{\textnormal{fcc}}}}
\def\chcp{\ensuremath{\mathcal{L}_{\textnormal{hcp}}}}

\def\conv{\ensuremath{\textnormal{conv}}}

\def\cub{\ensuremath{Q_\mathrm{co}}}
\def\oct{\ensuremath{Q_o}}
\def\czero{\ensuremath{\c\backslash\left\{0\right\}}}
\def\d{\ensuremath{\mathcal{D}}}
\def\diam{\ensuremath{\textnormal{diam}}}
\def\dist{\ensuremath{\textnormal{dist}}}

\def\eps{\ensuremath{\varepsilon}}

\def\g{\ensuremath{\mathcal{G}}}

\def\gl{\ensuremath{\hat{\Gamma}}}
\def\glab{\ensuremath{\Gamma}}

\def\half{\ensuremath{\frac{1}{2}}}
\def\I{\ensuremath{\textnormal{Id}}}

\def\K{\ensuremath{\mathcal{K}}}

\def\n{\ensuremath{\mathbb{N}}}

\def\O{\ensuremath{\Omega}}

\def\r{\ensuremath{\mathcal{R}}}

\def\R{\ensuremath{\mathbb{R}}}
\def\re{\R}

\def\rstar{\ensuremath{r^*}}

\def\t{\ensuremath{\mathcal{T}}}
\def\tcub{\ensuremath{Q_\mathrm{tco}}}

\def\u{\ensuremath{\mathcal{U}}}

\def\veps{\ensuremath{\varepsilon}}

\def\Xreg{X_\mathrm{reg}}
\def\Xcub{X_\mathrm{co}}
\def\Xtcub{X_\mathrm{tco}}

\def\z{\ensuremath{\mathbb{Z}}}

\def\Id{\mathrm{Id}}
\def\nnb{{\mathcal S}}
\def\sx{S^2}

\begin{document}

\title{Face-centered cubic crystallization of atomistic configurations}
\author{L. Flatley \and F.Theil
\thanks{\mbox{Corresponding author (f.theil@warwick.ac.uk)}}
\\
Mathematics Institute, University of Warwick, Coventry, CV47AL (UK)}
\date{}



\maketitle
\begin{abstract}
We address the question of whether three-dimensional crystals are
minimizers of classical many-body energies. This problem is of
conceptual relevance as it presents a significant milestone towards
understanding, on the atomistic level, phenomena such as melting or
plastic behavior. We characterize a set of rotation- and
translation-invariant two- and three-body potentials $V_2, V_3$ such
that the energy minimum of
$$ \frac{1}{\#Y}E(Y) = \frac{1}{\# Y}\left(2\sum_{\{y,y'\} \subset Y}V_2(y, y')
+ 6\sum_{\{y,y',y''\}\subset Y} V_3(y,y',y'')\right)$$
over all $Y \subset \R^3$, $\#Y = n$ converges to the energy per particle in
the face-centered cubic (fcc) lattice as $n$ tends to infinity.
The proof involves a careful analysis of the symmetry properties of the
fcc lattice.
\end{abstract}
\section{Introduction}
A material is crystalline if its underlying atomic
structure comprises a (multi)-lattice of particles.  It is known,
through techniques such as x-ray diffraction, that most materials
crystallize when at a sufficiently low temperature.
Whilst the range of crystalline materials is vast, however, the number of
underlying atomic structures is relatively small.  For example, more than half
of the metals crystallize into a face-centered cubic lattice (fcc), a
hexagonal-close packed crystal lattice (hcp) or a body-centered cubic lattice
(bcc). The question of why lattices are ubiquitous in nature can be
reformulated mathematically: why is it the case that so many
energy functionals admit periodic minimizers?

The premise of this paper is to provide a reasonably large set of
energy functionals $E: \R^{3 \times n} \to \R$ which are invariant under
translations and rotations and permutations for which the optimality of
unique periodic (lattice) configurations can \lf{be} rigorously proven in \lf{the} many particle limit $n \to \infty$.
Of course uniqueness of the optimal lattice only holds up to translations and rotations.

We will focus on cases where where the optimal lattice $\c$ is given by the fcc lattice
\begin{eqnarray}\label{eq:fccdefn}
 \cfcc:=(b_1 \ b_2 \ b_3)\z^3=\left\{\sum_{i=1}^3a_i b_i\; :\;a_i\in\z, \ i=1,2,3\right\}
\end{eqnarray}
\lf{where} $b_1=\frac{1}{\sqrt{2}}(0,1,1)^T,$
\ $b_2=\frac{1}{\sqrt{2}}(1,0,1)^T$ and
$b_3=\frac{1}{\sqrt{2}}(1,1,0)^T$ are column vectors.

To our best knowledge, the first three-dimensional examples for which the existence of
periodic minimizers can be guaranteed are due to A.~Suto, \cite{Suto05, Suto06}. The
results provide the existence of pair interaction potentials which
admit periodic ground states. However, the potentials are very delocalized
resulting in a significantly degenerate behavior of the ground states:
every periodic configuration with a sufficiently high density is a minimizer.
Our results are concerned with localized potentials where the
strength of the interaction between pairs of particles decays with an
inverse power law as a function of the distance. As a consequence we obtain that the minimum energy per particle converges to the energy of the fcc lattice as the number of particle tends to infinity.

The assumption that the potentials are localized induces a natural
splitting of the interaction energy into three separate contributions:
short-range, medium-range and long-range interaction energies. The analysis of pure short-range models
has a long history in discrete geometry. This approach was first
discussed by Kepler \cite{Kep}, who suggested that the six-fold
symmetry of snowflakes could be explained by considering ice
to be composed of small balls of vapor, packed together in
planes in the tightest way possible.  In two dimensions, this
tightest packing is an arrangement of discs whose centers are placed
at the points of a triangular lattice.  Such an arrangement clearly
exhibits the six-fold symmetry of a snowflake. A refinement of these
geometric ideas allows the construction of compactly supported
potentials for which it can be shown that minimizers
are translated, rotated and dilated subsets of the triangular lattice
\cite{Rad}. A more detailed analysis in \cite{FSY} even provides
a complete characterization of the surface energy and a proof that
after rescaling the minimizers converge to the Wulff shape as $n$
tends to infinity.

In three dimensions the situation is significantly more
involved. In particular, for pure short-range models, it cannot be
expected that minimizers are necessarily periodic. Illustrative examples are the theorems by
Hales \cite{Hal} and Sch\"utte and Van der Waerden \cite{SchVdWae}.
The results show that \textit{close-packed structures}, and in particular
fcc or hcp, solve both the densest packing problem and the kissing problem in three dimensions.

The close-packed structures are constructed as follows:
consider a triangular arrangement of spheres of diameter $1$ in a
two-dimensional plane $A=(b_1,b_2)\mathbb{Z}^2 \subset \R^3$. Now add a translated
copy $B=A+b_3$ so that each sphere of B sits
directly above a hole of $A$.  For the third layer $C$, there are two
possibilities: either \amend{1}{$C=A+\frac{2\sqrt{2}}{3}(1,1,-1)^T$} and $C$ lies directly
above $A$; or $C=A+2b_3$ and each sphere of C
sits directly above a hole of both $A$ and $B$.  In the first instance, we
relabel $C=A.$ By repeating this stacking construction we obtain a close-packed
structure.  The fcc lattice and hcp lattice are defined to be
the close-packed structures with stacking sequences $ABCABCABCA...$ or
$ABABABABA...,$ respectively.  This definition of fcc is equivalent to
(\ref{eq:fccdefn}), whilst hcp is a multi-lattice which is generated by the basis
$\tilde b_1=b_1,\tilde b_2=b_2,\tilde b_3=\frac{2\sqrt{2}}{3}(1,1,-1)^T$
and a translation vector $t=b_3$ in the sense that
$$ \chcp = \{0,t\}+(\tilde b_1\,\tilde b_2\,\tilde b_3)\mathbb{Z}^3.$$
  \begin{figure}[H]
\begin{center}
\includegraphics[scale=0.2]{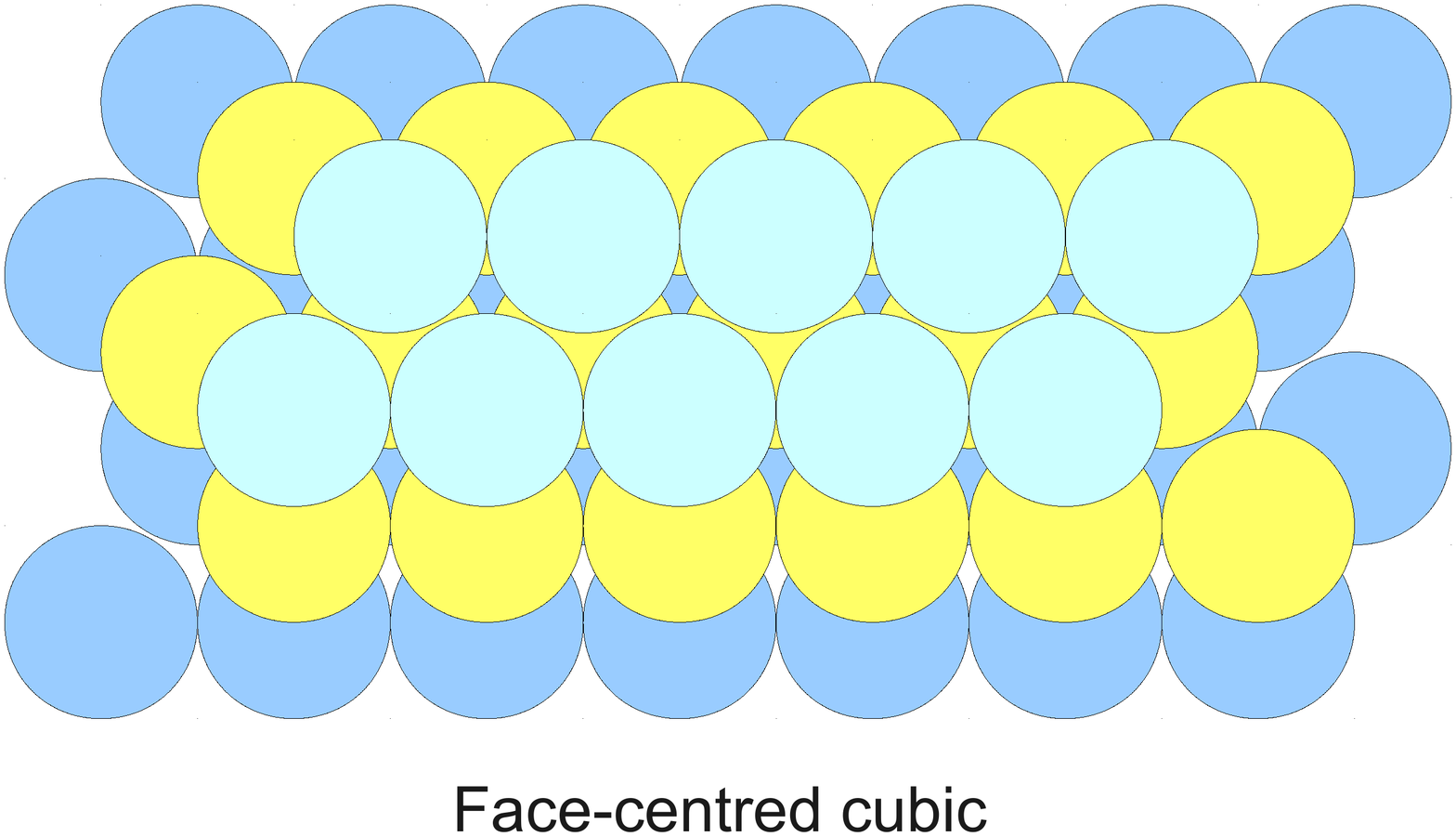}
\includegraphics[scale=.2]{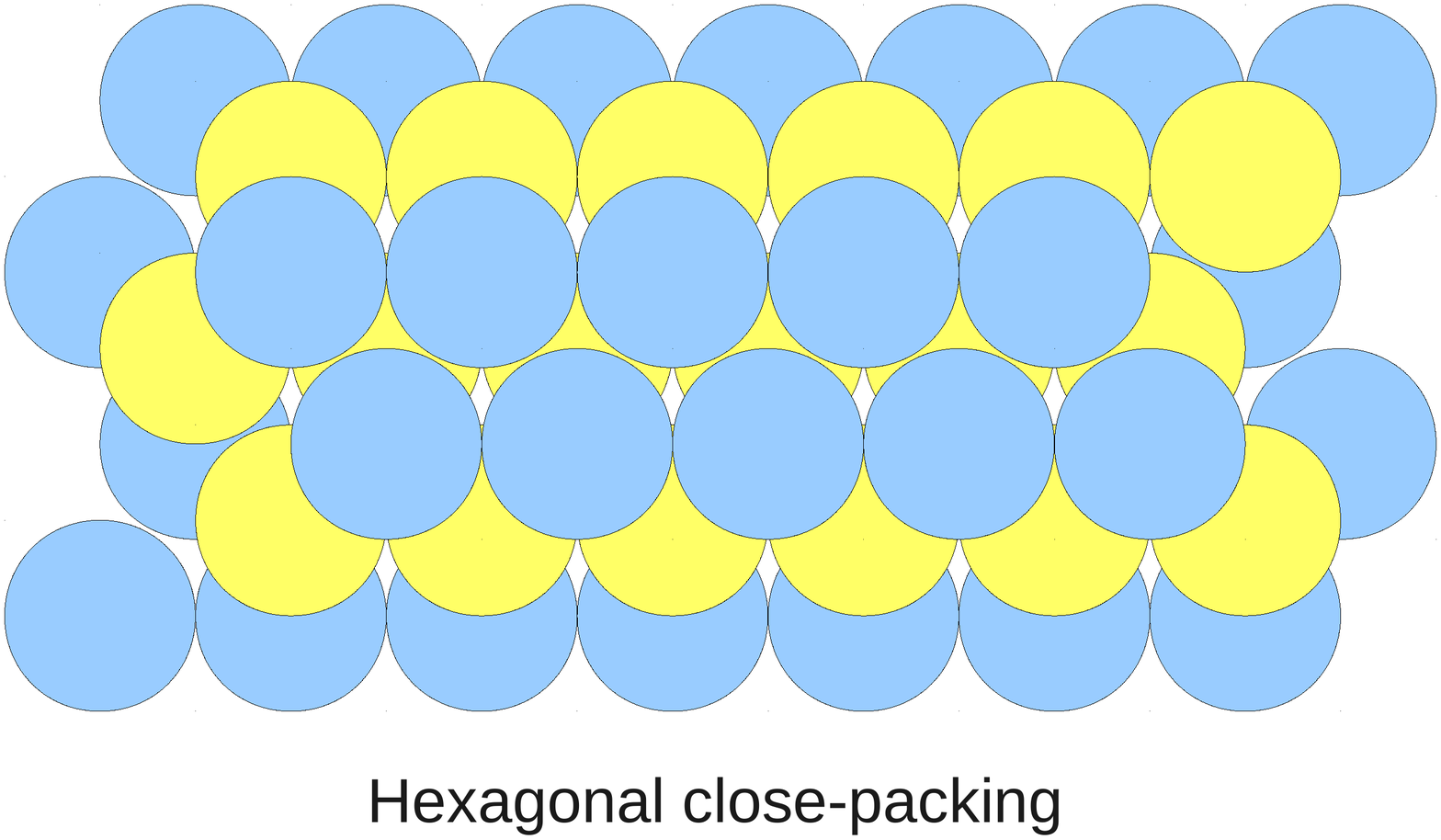}
\label{fig:threelayers}\caption{Three layers of the fcc and hcp crystal lattice}
\end{center}
\end{figure}
For our purposes, the most relevant property of the close-packed
structures is that they maximize the number of nearest neighbors to
every point in the packing, i.e. they solve the kissing problem.
It has been known since 1953 that the solution of the three-dimensional
kissing problem is twelve \cite{SchVdWae, Lee,Mus}, and
therefore
these maximal neighborhoods each contain twelve points.  In our analysis, it is
this common feature of the close-packed structures that makes them the
most natural candidates for crystallization.

The degeneracy of purely local models regarding the crystalline structure is
broken by the presence of long range interactions which occur naturally in
physically relevant situations \cite{Len}. The main result of this paper is
Theorem \ref{thm:mainintro}, which states the existence of a large set of
rotation and translation-invariant interaction potentials $V_2,V_3$ with the
property that the ground states behave asymptotically
like an fcc lattice energy, as the number of particles tends to infinity.

\begin{thm}[Main theorem]\label{thm:mainintro}
There exists $\alpha_0>0$ such that for any
  $\alpha\in(0,\alpha_0)$ and any pair of $\alpha$-localized potentials $(V_2,V_3)$ (cf. Definition~\ref{defn:adpot}) which is invariant under translations
and rotations \lf{in} the sense that
\begin{align} \label{symmetries}
V_i(R\xi_{\pi(1)} +t, \ldots, R\xi_{\pi(i)} + t) = V_i(\xi_1,\ldots,\xi_i)
\text{ for all permutations } \pi,\; R \in O(3),\; t \in \R^3,\; i\in \{2,3\}
\end{align}
and any finite set $Y \subset \R^3$ the inequality
\begin{align} \label{lb}
\frac{1}{\# Y}\,E(Y)>\min_{r>0}\Efcc(r)
\end{align}
holds, where
\begin{align} \label{eq:energy}
E(Y)&= 2\,\sum_{\{y,y'\} \subset Y} V_2(y,y') +
6\,\sum_{\{y,y',y''\} \subset Y} V_3(y,y',y''),
\end{align}
and
\begin{align}
\Efcc(r)&= \sum_{y \in \cfcc \setminus\{0\}} V_2(0,r\,y)
+2\,\sum_{\{y,y'\} \subset \cfcc \setminus\{0\}} V_3(0,r\,y,r\,y'). \nonumber
\end{align}
Moreover,
\begin{align} \label{upperbound}
\lim_{R \to \infty} \frac{1}{\# Y_R} E(Y_R) = \min_{r>0}\Efcc(r),
\end{align}
if $Y_R = B(0,R) \cap \cfcc$. $B(0,R)$ denotes the ball with radius $R$ centered at
the origin.
\end{thm}
Note that \lf{we later assume that $\min_{r>0} \Efcc(r) = \Efcc(1).$}  This
assumption does not involve any loss of generality as it can be achieved by rescaling the
potentials $V_i$.

The role of the three-body potential $V_3$ is of solely \lf{a} technical nature. It
is quite likely that the assumptions of Theorem~\ref{thm:mainintro} can be
relaxed so that also pure pair models are covered, \lf{(}see the discussion
in Section~\ref{sec:discussion}\lf{)}.

It
cannot be expected that a similar theorem holds if $\cfcc$ is replaced by $\chcp$. For any dilation parameter $r>0$ and generic potentials $V_2$ and $V_3$ the dilated lattice $r\chcp$ is {\em not} optimal.
To see this we define for $\c\in \{\cfcc,\chcp\}$ the stored energy function
$$ W_{\c}(F)= \sum_{k \in \c\setminus\{0\}} V_2(0,F\,k) +2\,\sum_{\{k,k'\}\subset\czero}V_3(0,F\,k,F\, k'),$$
and the Piola-Kirchhoff tensor $S(F) \in \R^{3 \times 3}$
by
$$ S_{ij} = \frac{\partial W_{\c}}{\partial F_{ij}}(F).$$
If $W_{\c}(\rstar\, \Id) = \min_{r>0}W_{\c}(r\, \Id)$, then
\begin{align} \label{eq:trace}
\mathrm{trace}\, S(\rstar \, \Id)=0.
\end{align}
It is well known that $S(F)$ is a multiple of the identity if $F$ is a multiple of the identity and $\c = \cfcc$; together with (\ref{eq:trace}) this implies that $S(\rstar\,\Id) =0$.
For the convenience of the reader we give a short proof.

Observe that $W$ is invariant under the combined action of $O(3)$ and
the point group
$$G =\{ g \in O(3) \; : g\c = \c\}$$
in the sense that
\begin{align} \label{eq:inv}
W_{\c}(R\,F\,g)= W_{\c}(F) \text{ for all } R \in O(3), g \in  G.
\end{align}
Indeed,
\begin{align*}
W_{\c}(R\,F\,g) =& \sum_{k \in \c\setminus\{0\}} V_2(0,F\,g\,k) +2\,\sum_{\{k,k'\}\subset\czero}V_3(0,F\,
g\,k,F\,g\, k')\\
=&
\sum_{k \in \c\setminus\{0\}} V_2(0,F\,k) +2\,\sum_{\{k,k'\}\subset\czero}V_3(0,F\,k,\,F\, k') = W_{\c}(F).
\end{align*}
The penultimate equation is a consequence of (\ref{symmetries}), the final equation holds because $g\,\c=\c$.

The invariance (\ref{eq:inv}) implies that $S(F)$ is equivariant,
$$ S(R\,F\,g) = R\, S(F)\,g \text{ for all } R \in O(3) \text{ and } g \in G.$$
The choice $F = r\,\Id$ and $R=g^T$ delivers the equation
\begin{align} \label{eq:equiv} \Sigma = g^T \Sigma\, g \text{ for all } g \in G,\end{align}
where $\Sigma= S(r\,\Id)$.

A more general discussion of equivariant maps can be found in \cite{GSS}, here we
only present a self-contained proof of this rather simple example.
Eqn.~(\ref{eq:equiv}) implies that the restriction of $\Sigma$ to
the span of $\{b_1,b_2\}$ is a multiple of the identity if $\c \in \{\cfcc, \chcp\}$.
Indeed, if $G$
contains the reflection $g_v= \Id - 2\,v \otimes v$ for some $v \in \R^3$ such that $|v|=1$, then $v$ is an eigenvector of $\Sigma$. It can be checked that
$g_v \in G \text{ if } v \in \{b_1, b_2, b_2-b_1\}$, hence $b_1$, $b_2$ and $b_2-b_1$ are eigenvectors.
As these three vectors are linearly dependent the eigenvalues coincide. In the case $\c=\cfcc$
the same argument can be applied to $\{b_1, b_3, b_3-b_1\}$ and one obtains
the existence of $\lambda(r)\in \R$ such that
\begin{align} \label{eq:simple} S(r\,\Id) = \lambda\, \Id.
\end{align}
On the other hand,
if $\c = \chcp$ then every matrix $\Sigma$ which can be written as
\begin{align*}
\Sigma = \lambda \tilde b_3 \otimes \tilde b_3 + \mu \left(|\tilde b_3|^2 \, \Id -
\tilde b_3 \otimes \tilde b_3\right)
\end{align*}
is compatible with (\ref{eq:equiv}). In this case eqn.~(\ref{eq:trace}) only implies that $\lambda=
-2 \mu$. Thus, the configuration $y(x)= F_\eps x$ with $F_\eps = \rstar \,\Id - \eps \Sigma$ has lower energy if $\mu \neq 0$ and $0<\eps\ll 1$.

Theorem \ref{thm:mainintro} provides a significant generalization of the
two-dimensional result in
\cite{The}. The differences affect both the analysis of the local interactions and the analysis of the long-range interactions. In particular,
\begin{enumerate}
\item The challenges met by the local analysis are considerably more
  involved. For example, solutions of the kissing problem in three dimensions
  are highly degenerate.  Unlike in the two-dimensional setting, it is not
  possible to identify kissing configurations as orbits of simple symmetry
  groups.
\item To differentiate between the fcc and the hcp lattice it is unavoidable
  to consider medium-range interactions whose range reaches beyond nearest neighbors.
  Specifically, we require an analysis of second and third nearest
  neighbor interactions.
\item The heart of the proof consists of localizing long-range interaction
  energies.  The analysis in \cite{The, LiE} is limited to highly symmetric
  two-dimensional lattices and cannot be generalized to three dimensions where the
  lattices are less symmetric.
\end{enumerate}
Here we deal only with energetic crystallization.  Positional crystallization
can be obtained under suitable boundary conditions as a corollary to
Theorem~\ref{thm:mainintro} and a generalization of \cite{The}.
Corollary~\ref{cor:per} provides two such results when the particle positions are subjected
to clamped and periodic boundary conditions.
\ft{\begin{defn}
Let $Z \subset \R^3$ be countable. If there exists a Bravais lattice $\c \subset \R^3$ such that
$$Z+\eta= Z \text{ for each } \eta \in \c$$
or
$$ \#(\c \setminus Z)+\#(Z\setminus \c) < \infty,$$
we say that $Z$ is $\c$-periodic or a compactly supported perturbation of $\c$.\\[0.3em]
If $Z$ is $\c$-periodic or a compactly supported perturbation of $\c$, then the energy of $Z$ is defined by
 \begin{eqnarray} \label{defEA}
  E_{\a}(Z):=\sum_{y \in Y}\left(\sum_{y' \in Z}
  V_2(y,y')
+\sum_{\atop{y',y'' \in Z}{y\neq y' \neq y'' \neq y}}V_3(y,y',y'')\right),
\end{eqnarray}
where $Y = Z/\c$, $\a = Z \setminus Y$ if $Z$ is periodic, or $Y = Z\setminus \c$, $\a = \c \cap Z$ if $Z$ is a compactly supported perturbation of $\c$.
\end{defn}}
Observe that $E_{\a}(Z) =0$ if $Z \subset \c$.
\amend{2}{
\begin{cor}[Ground states with periodic or clamped boundary conditions]\label{cor:per}
  Let $\c$ be a Bravais lattice. There exists $\alpha_0>0$ such that for any $\alpha$-localized pair $(V_2,V_3)$ and any $Z \subset \R^3$ which is
either $\c$-periodic or a compactly supported perturbation of $\c$, the inequality
\begin{eqnarray*}
\frac{1}{\#Y} \,E_{\a}(Z) \geq \Efcc(1)
\end{eqnarray*}
holds with the convention $\# Y=1$ if $Y = \emptyset$.\\[0.3em]
Equality is attained if and only if $Z$ is periodic and
there exists a translation vector $t\in\re^3$ and a rotation $R \in SO(3)$ such that
\begin{eqnarray*}
 R\,Z+t=\cfcc.
\end{eqnarray*}
\end{cor}
}


Before outlining the proof we comment on possible physical applications.
Classical groundstates are limiting cases of more general states.
One important group of examples is given by quantum mechanical energies where the
states are many-body wavefunctions $\psi \in L^2(\R^{3 \times n}).$

Another generalization of classical groundstates
are Gibbs states at finite temperature. Here one is interested in the properties
of the probability measure
$$ P_{\rho, \beta, \Lambda}(y) = \mathrm{e}^{-\beta\,E_n(y)+F(\beta,\rho,\Lambda)},$$
where $\Lambda \subset \R^3$, $y \in \Lambda^n$, $n = \rho\,|\Lambda|$, $\beta>0$ and $F(\beta,\rho,\Lambda)$ is
a normalization constant. The thermodynamic limit is obtained by studying the
asymptotic behavior of $P_{\rho,\beta,\Lambda}$ as $n\to \infty$.
 If $\rho$ is close to the density selected by the groundstates and the inverse temperature $\beta$ is large, then it is expected that $P_{\rho,\beta}$ exhibits long-range positional order in the following sense: There exists $s(\rho,\beta)>0$ such that
\begin{align} \label{lrc}
\lim_{n \to \infty} \frac{1}{n-1}\int \dd P_{\rho,\beta,\Lambda}(y)\,\sup_{Q \in O(3)}\sum_{i=2}^n f\left(s\, Q\,(y_i-y_1)\right) > 0
\end{align}
for any \lf{$\cfcc$}-periodic and non-constant function $f \in C(\R^3)$ with average 0.
The Mermin-Wagner theorem \cite{MW} states that (\ref{lrc}) does not hold for
two-dimensional systems.

It would be highly desirable to link the properties of the groundstates with the properties
of the Gibbs states.

\subsection{Outline of the proof of Theorem \ref{thm:mainintro}}
In Definition~\ref{defn:adpot} it is assumed that $V_2$ and $V_3$ are suitably
normalized so that $\min_{r>0}\Efcc(r) = \Efcc(1)$ and we define
$$ \estar = \Efcc(1).$$
The upper bound (\ref{upperbound}) is an easy consequence of Theorem~1 in \cite{BLL}.

Our focus, therefore, is to establish the lower bound (\ref{lb}).
For the proof it is advantageous to view \lf{$y\in Y$} as a map.  Therefore we assume that there
exists an index set $X$ such that $\# X = \# Y$ and
$$ Y = \{y(x) \; : \; x \in X\}.$$
Then
\lf{\begin{align} \label{eq:vectnot}
 E(Y)=E(y) = 2\sum_{\{x,x'\}\subset X} V_2(y(x),y(x')) +
6\sum_{\{x,x',x'' \}\subset X} V_3(y(x),y(x'),y(x'')).
\end{align}}
Central to the proof of (\ref{lb}) is the set of (nearest neighbor) edges,
defined by
\begin{align}\label{eq:sdefn}
 \amend{8}{\nnb}:=\left\{(x,x')\in X \times X\; :\;\left|\left|y(x')-y(x)\right|-1\right|
\leq\alpha\right\}.
\end{align}
By construction, $\nnb$ is induced by $y,$ although this dependency is not shown.
The structure induced by $\nnb$ allows us to define defects: A label $x$ is an element of $X_\mathrm{co}$ if the nearest neighbors of $x$ can be mapped
bijectively to the nearest neighbors of the origin in \lf{$\cfcc$} such that nearest
neighbors pairs are mapped to nearest neighbor pairs. The complement of $X_\mathrm{co}$
is called the set of {\em defects} or boundary $\partial
X=X\setminus X_\mathrm{co}.$

The main task is to establish the
the finer estimate
\begin{align} \label{finebound} E(y)\geq \estar\,\#X+C\sum_{(x,x')\in
    \nnb}\left|\left|y(x')-y(x)\right|-1\right|^2+ C\alpha^\frac{1}{2}\#\partial X,
\end{align}
where $C>0$ depends on \amend{3}{$V_2$ and  $V_3$. Inequality (\ref{finebound}) not only implies (\ref{lb}), it also provides insight into the additional structure offered by $\nnb$. We will demonstrate that $C$ can be chosen independently of $\alpha$ and $\# X$.}

The proof of~(\ref{finebound}) is organized around three key concepts: a
geometric analysis of nearest neighbors (c.f. Proposition~\ref{prop:maxnhd});
the construction of a reference configuration which identifies large
\textit{defect-free} patches of the ground state with rotated and translated
subsets of the lattice \lf{$\cfcc$} (c.f.  Proposition~\ref{prop:refconfig}); a
resummation of the energy to recover a lattice energy together with several error
terms (c.f. Section~\ref{sec:energy}).

In Section~\ref{sec:adpot}, we introduce the set of $\alpha-$localized
potentials and derive a lower bound on the distance between ground state
particles (Proposition~\ref{prop:mindist})).

In Section \ref{sec:gsprop}, we characterize the fcc lattice based on
local properties. This characterization allows us to identify arbitrarily
large defect free patches of the ground state which can be identified with subsets of \lf{$\cfcc$} via discrete imbeddings.
Existence results for imbedding are stated in Section~\ref{sec:refconex}.

In Section \ref{sec:rigidity}, we demonstrate the long-range rigidity of the
ground state.  The main result here is Proposition~\ref{prop:distsum}, which
states that the $L^2$ proximity of a ground state deformation gradient to a
rigid rotation is controlled by a quadratic sum of edge length distortions.

In Section~\ref{sec:pairs}, we introduce path sets, they will be used
to bound the interaction energy from below by local expressions.

In Section~\ref{sec:energy} the proof of Theorem~\ref{thm:mainintro} is given.
It relies on the concept of path sets to localize
the long-range interactions. We obtain two types of error terms: energy
contributions arising from the
geometric distortion of bonds; and a surface energy contribution arising from
the omission of individual bonds.  The first error term is controlled by the
rigidity estimates of Proposition~\ref{prop:distsum}, which reduce the
long-range energy contributions to a quadratic sum of edge length distortions.
The second error term is controlled by size of the set of defects.

The main focus of the paper is on the analysis of low energy states of a large
number of particles $\#X$. Since the asymptotic behavior
is discussed only at the end, we suppress any dependency
on $\# X$ in the notation. The letter $C$ \delete{$c$} always denotes a generic positive
constant, which depends on $V_2$ and $V_3$, but not on $\alpha$ or $\#X,$ provided $\alpha$ is
sufficiently small, and whose value may vary from line to line.  A glossary of
the most commonly used notation is included at the end of the appendix.

\section{Admissible potentials}\label{sec:adpot}
It is easy to see that the invariance~(\ref{symmetries})
implies the existence of functions $V:[0,\infty) \to \R$, $\Psi:[0,\infty)^3 \to \R$ such that
\begin{align} \label{potrep}
V_2(y_1,y_2) = V(|y_2-y_1|) \text{ and } V_3(y_1,y_2,y_3)=\Psi(|y_2-y_1|, |y_3-y_2|, |y_1-y_3|) .
\end{align}
The set of admissible potentials of Definition \ref{defn:adpot} are characterized
by a small parameter $\alpha>0.$  The pair potentials are chosen to have growth
behavior similar to that of a Lennard-Jones potential, whilst the three-body
potentials take a generalized form of the Stillinger-Weber potential.
\begin{defn}[$\alpha-$localized potentials]\label{defn:adpot}
Let $\alpha>0$ be a positive parameter.  We say that $(V_2,V_3)$ are
$\alpha$-localized if there exist potentials $(V,\Psi)$ in $Y_\alpha$
such that (\ref{potrep}) holds.
The set $Y_\alpha \subset C^1([0,\infty)) \times C^1
\left([0,\infty)^3\right)$ is defined by the following requirements.
\begin{enumerate}
\item The pair potential $V$ has the properties
$\lim_{r \to \infty} V(r) = 0$, $V$ is normalized in the sense that $V(1)=-1$,
\begin{align} \label{eq:vnorm}
\min_{r>0} \sum_{k \in \lf{\cfcc} \setminus \{0\}} (V(r\,|k|)-V(|k|))=0,
\end{align}
and satisfies the conditions,
\begin{align}
V(\sqrt{8/3})- 3\,V(\sqrt{3}) &\geq \alpha^\frac{1}{2}, \label{fcccond}\\
V(r) &\geq\frac{1}{\alpha}  \text{ for } r\in\left[0,1-\alpha\right],
\label{assump:vone}\\
V''(r) & \geq 1  \text{ for } \ r\in(1-\alpha,1+\alpha),
\label{assump:vtwo}\\
V'(\sqrt{3}) &\geq 0, \label{Vpsign}\\
\label{assump:vfourprimetwo} |V''(r)| & \leq\alpha^\frac{1}{4} \text{ for } r\in
\left[1+\alpha,\sqrt{7/2}\right], \\
\left|V''(r)\right| & \leq \alpha r^{-10} \text{ for }  r\in\left[\sqrt{7/2},\infty\right).\label{assump:vfive}
\end{align}
\item The three-body potential $\Psi$ has the properties
\ft{\begin{align} \label{tbc}
&\min_{r_1,r_2,r_3\geq 0}\Psi(r_1, r_2, r_3) = \Psi(1,1,1)=-1 \text{ and } \Psi(r_1, r_2, r_3) \geq 0 \text{ if } \max_{i}|r_i-1|\geq \alpha,\\
&\Psi(r_1, r_2, r_3) \geq \frac{1}{\alpha} \text{ if } \min_{i}r_i \leq 1-\alpha \text{ and} \max_{i}r_i< \frac{4}{3}, \label{tbl}\\
&\Psi(r_1, r_2, r_3)=0 \text{ if } \max_i r_i \geq \frac{7}{5}. \label{comp-supp}
\end{align}}
\end{enumerate}
\end{defn}
\subsection{Discussion} \label{sec:discussion}
A heuristic argument behind the choice of
admissible potentials is as follows: Assumption (\ref{eq:vnorm}) sets the lattice parameter to 1.  This assumption simplifies the notation and does not involve a loss of generality. The large energies of
short-range bonds created by condition (\ref{assump:vone}) ensures a minimum
distance of $1-\alpha$ between particles (c.f. Proposition~\ref{prop:mindist}).
Conditions (\ref{assump:vone})-(\ref{assump:vtwo}) on $V$
create a sharp, prominent well close to $r=1,$ which favors configurations
which maximize the number of nearest neighbor pairs.
Assumption~(\ref{fcccond}) selects fcc as the
optimal crystalline form, since
\begin{align*}
-2 &= \#\left\{k \in \cfcc \; : \; |k|= \sqrt{8/3}\right\}-
\#\left\{k \in \chcp \; : \; |k|= \sqrt{8/3}\right\},\\
6 &= \#\left\{k \in \cfcc \; : \; |k|= \sqrt{3}\right\}-\#\left\{k \in \chcp \; : \; |k|= \sqrt{3}\right\}.
\end{align*}

Assumption~(\ref{Vpsign}) is of purely technical
nature and is satisfied for Lennard-Jones type potentials.
Assumptions (\ref{assump:vfourprimetwo}) and (\ref{assump:vfive}), which characterize the decay of $V$, entail that medium and long-range interactions are much weaker than the short-range
interactions.

The three-body potential $\Psi$ selects ground states which
maximize the number of edges in each nearest neighborhood. This serves to
reduce the number of nearest-neighbor graph structures down to just
two: the fcc and the hcp crystal lattices (c.f. Theorem~\ref{thm:HTT}).

The main part of the analysis concerns the pair energies.
The role of the three-body potential
is to geometrically determine the optimal crystalline form, by assigning
positive energy contributions to ground states which do not approximate fcc or hcp structures.

The assumptions on the pair potential $V$ are generic in the sense that there exist open subsets $\hat Y_\alpha$ of the weighted space \ft{$C^1_\rho([0,\infty))$}
such that each $V \in \hat Y_\alpha$ satisfies assumptions (\ref{eq:vnorm}) - (\ref{assump:vfive})
after rescaling. On the other hand, the assumptions on the three-body potential $\Psi$ are not generic, i.e. the set of potentials which satisfy (\ref{tbc}), (\ref{tbl}) after rescaling does not contain an open set. The proof of Theorem~\ref{thm:mainintro} can be generalized if the conditions are slightly relaxed so that $Y^\alpha$ is open; this would involve a significant increase of the notational complexity. \ft{Similarly, it is possible to relax \eqref{assump:vtwo} and \eqref{assump:vfive} at the expense of additional assumptions so that $\hat Y_\alpha \subset C^2_\rho$.}

It is conceivable that the dependence of
Theorem~\ref{thm:mainintro} on $V_3$ can be omitted entirely. Although this remains
an open problem, the following conjecture provides a possible route to
eliminate the necessity of $V_3.$
\begin{conj}\label{conj:fcchcp}
Let $Z\subset\re^3$ satisfy that $\left|z'-z\right|\geq 1$ for all $\{z,z'\}\subset Z$ and, for every $z\in Z,$ let $N(z):=\left\{z'\in Z:\left|z'-z\right|=1\right\}.$  If $z,z'$ have the properties $\# N(z)= \# N(z') = 12$ and $z \in N(z')$, then $\#(N(z)\cap N(z'))\geq 4$.
\end{conj}
Together with Theorem \ref{thm:HTT}, Conjecture \ref{conj:fcchcp} implies that $\#(N(z)\cap N(z'))=4$ for all $z,z'\in Z$.
It is not hard to see that up to rotation there are only two subsets of $\sx$ with 12 points such that each point has precisely 4 neighbors:
the cuboctahedron and the twisted cuboctahedron \lf{(defined by eqn. (\ref{eq:cub})).}

\subsection{Results concerning admissible potentials} \label{adres}
For every $r>0,$ the \textit{renormalized energy} $\Efcc(r)$ assigns an average energy per particle to the homogeneously deformed lattice \lf{$r\cfcc.$}

\begin{defn}\label{defn:renorm}
Let $(V,\Psi) \in Y_\alpha$ for some $\alpha>0$. The associated
renormalized pair potential  $V^*$ is defined by
\begin{eqnarray}\label{eq:renormdefn}
V^*(r):= \lf{\sum_{k\in\cfcc\setminus \{0\}}}V(r\left|k\right|) \ \textnormal{for all} \ r>0.
\end{eqnarray}
Recall also the definition $$\Efcc(r):=V^*(r)+2\sum_{\{y,y'\}\subset\cfcc\setminus\{0\}}\Psi(r|y|,r|y'|,r|y-y'|).$$
We call $\Efcc$ the \textit{renormalized energy per particle}.
\end{defn}
Note that assumptions~(\ref{eq:vnorm}) and (\ref{tbc})
and the fact that
$$\amend{4}{ \#\{(k_1,k_2)\in \cfcc \times \cfcc \; : \; |k_1|=|k_2| = |k_2-k_1|=1\} = 48}$$
imply the identity
$$ \min_{r} \Efcc(r) = \Efcc(1)=V^*(1) + 48\, \Psi(1,1,1).$$
\begin{lemma}[Equilibrium condition]\label{lemma:renormmin}
If $(V,\Psi) \in Y_\alpha$ then
\begin{eqnarray}\label{eq:renormmin}
\lf{\sum_{k\in\cfcc\setminus\{0\}}}\left|k\right|V'(\left|k\right|)=0.
\end{eqnarray}
\end{lemma}
\begin{proof}
The normalization assumption (\ref{eq:vnorm}) implies that $(\Vstar)'(1)=0$. This is
(\ref{eq:renormmin}).
\end{proof}
\begin{lemma}\label{lemma:assumptions}
Let $(V,\Psi) \in Y_\alpha$.
There exists $\alpha_0>0$ such that for every $\alpha\in(0,\alpha_0)$, $r>0$ the following estimates hold:
\begin{align}
 |V'(r)| & \leq \alpha^\frac{1}{4} \text{ for } r\in \left[1+\alpha, \sqrt{7/2}\right], \label{assump:Vlip}\\
|V(r)| &\leq \alpha^\frac{1}{4} \text{ for } r \in \left[1+\alpha,\sqrt{{7}/{2}}\right], \label{assump:vzero}\\
\left|V(r)\right| &\leq \alpha r^{-8} \ \textnormal{for all} \ r\geq \sqrt{7/2},\label{eq:assumptwo}\\
V(r) & \geq-2 \ \textnormal{for all} \ r\geq 0.\label{eq:assumpthree}
\end{align}
\end{lemma}

\begin{proof}
The proof follows immediately from assumptions (\ref{eq:vnorm})-(\ref{assump:vfive}).
\end{proof}

\begin{prop}[Minimum distance bound]\label{prop:mindist}
Let $(V,\Psi) \in Y_\alpha$. There exists $\alpha_0>0$ such that if $\alpha\in(0,\alpha_0)$ then any ground state $y:X\rightarrow\re^3$ of the associated energy (\ref{eq:energy}) satisfies the minimum distance bound
\begin{eqnarray}\label{eq:mindist}
\min_{\atop{x,x'\in X}{x\neq x'}}\left|y(x')-y(x)\right|>1-\alpha.
\end{eqnarray}
\end{prop}
\begin{proof}
Let $M:=\max_{\eta\in\re^3}\#\left(y(X)\cap B(\eta,\half(1-\alpha))\right)$ and assume wlog that the maximum is achieved at $\eta=0.$  Set $B_M:=B(0,\half(1-\alpha))$ and $\mathcal{A}:=y^{-1}(B_M).$  We aim to show that $M=1.$  Assumption (\ref{assump:vone}) implies
\begin{eqnarray}\label{eq:mindistone}
 \sum_{\atop{x,x'\in \mathcal{A}}{x\neq x'}}V(\left|y(x')-y(x)\right|)\geq\frac{1}{\alpha}M(M-1).
\end{eqnarray}
By moving the positions $y(\mathcal{A})$ to infinity in such a way that their mutual distances diverge, \lf{we obtain}
\begin{align}
 &\sum_{\atop{x\in\mathcal{A}}{x'\in X\backslash\mathcal{A}}}V(\left|y(x')-y(x)\right|)
 +\sum_{\atop{x\in \mathcal A}{(x_1,x_2)\in X^2\setminus\mathcal{A}^2}} \Psi(|y(x)-y(x_1)|,|y(x)-y(x_2)|, |y(x_1)-y(x_2|). \\
 \leq&-\frac{1}{2\alpha}M(M-1)\label{eq:mindistancetwo}
\end{align}
For each $d\geq 0,$ let $\mathcal{T}(d):=y^{-1}\left(2(d+1)\,B_M\backslash 2d\,B_M\right)$ and $n(d):=\#\mathcal{T}(d).$
If $\amend{5}{1\leq d\leq 3}$, then (\ref{eq:assumpthree}) implies
\begin{eqnarray}\label{eq:mindistfour}
 \sum_{\substack{x\in\mathcal{A}\\ x'\in\mathcal{T}(d)}}V(\left|y(x')-y(x)\right|)\geq-2n(d)M.
\end{eqnarray}
For the long-range interactions the decay estimate~(\ref{eq:assumptwo}) implies for
sufficiently small $\alpha_0>0$, $\alpha \in (0,\alpha_0)$ and $d\geq 4$ that
\begin{eqnarray}\label{eq:mindistthree}
 \sum_{x\in\mathcal{A}}\sum_{x'\in\mathcal{T}(d)}V(\left|y(x')-y(x)\right|)\geq-C\,n(d)\,M\alpha
 \left((d-1)(1-\alpha)\right)^{-8}.
\end{eqnarray}
There exists a constant $C>0$ such that for any $d\geq 0,$ \ $\t(d)$ can be covered by $C(d+1)^2$ translated copies of $B_M$, which implies that $n(d)\leq CM(d+1)^2.$  Consequently.
\begin{eqnarray}\label{eq:mindistthree2}
 \sum_{x\in\mathcal{A}}\sum_{x'\in\mathcal{T}(d)}V(\left|y(x')-y(x)\right|)\geq -CM^2\left(\sum_{d=1}^3(d+1)^2+\frac{\alpha}{(1-\alpha)^8}\sum_{d=4}^{\infty}\frac{(d+1)^2}{(d-1)^8}\right).
\end{eqnarray}

To study the three-body interactions we define
$$
  e_3(x):=2\sum_{\{x_1,x_2\}\subset  X\setminus\{x\}}V_3(y(x),y(x_1),y(x_2)).
$$
and $K>0$ by the requirement that the set
$$ \{z:1-\alpha\leq |z|\leq 1+\alpha\}\subset \R^3$$ can be covered with $K$ translated copies of $B_M$.
As we are interested in the cases where $\alpha<1$ the constant $K$ can be chosen independently from $\alpha$. Assumptions (\ref{tbc}) and (\ref{tbl}) imply that
\amend{6}{\begin{align*}
e_3(x)\geq&\sum_{x'\in N(x)}\biggl( \Psi(1,1,1)\#\{x'' \; : \; (x,x''), (x',x'') \in \nnb\}\\
&\hspace{3cm}+\frac{1}{\alpha}\#\{x'' \; : \; y(x'') \in \mathcal A \setminus y(x)\}\biggr)\\
\geq& \sum_{x' \in N(x)}\left( K\,M\,\Psi(1,1,1)\,+ \frac{1}{\alpha}(M-1)\right).
\end{align*}}
If $\alpha < (2\,K)^{-1}$ and $M>1$, then $e_3(x)>0$.

 Consequently, comparing (\ref{eq:mindistthree2}) with (\ref{eq:mindistancetwo}), we obtain
\begin{eqnarray}\label{eq:mindistfive}
 \amend{7}{-CM^2\left(\sum_{d=1}^3(d+1)^2+\frac{\alpha}{(1-\alpha)^8}\sum_{d=4}^{\infty}\frac{(d+1)^2}{(d-1)^8}\right)\leq-\frac{1}{2\alpha}M(M-1).}
\end{eqnarray}
Since the left-hand side remains bounded as $\alpha$ tends to 0, we deduce that (\ref{eq:mindistfive}) can only hold for all $\alpha\in (0,\alpha_0)$ if $M=0$ or $1.$  Since $X$ is non-empty, we conclude that $M=1.$

\end{proof}

\section{Discrete reference configurations}\label{sec:gsprop}
A key step towards the proof of (\ref{finebound}) is the development of the concept of
a discrete reference configuration, which allows us to identify
parts of the configuration $\{y(x) \; | \;  x \in X\}$ as images of maps
$u: \omega \to \R^3$ with $\omega \subset \lf{\cfcc}$. We require a characterization of the crystal lattices $\cfcc$
and $\c_\mathrm{hcp}$ which is based on local properties of
the point configuration.
Neither eqn.~(\ref{eq:fccdefn}) nor the stacking sequence are useful
for our purposes.

\subsection{Nearest neighborhood geometry}\label{sec:localstructure}

\subsubsection{Local geometry of the fcc and hcp crystal lattices}
\label{sec:locallattice}
\amend{8}{Throughout this paper, we denote by $\sx\subset\re^3$ the unit sphere centered at the origin.}  The cuboctahedron and twisted cuboctahedron are defined by the relations
\begin{align}\label{eq:cub}
 \cub:=\cfcc\cap \sx  \text{ and } \tcub:=\chcp\cap \sx.
\end{align}
The surfaces of both convex hulls consist of twelve vertices,
twenty-four edges and fourteen faces, eight of which are equilateral triangles and six
of which are squares. Notice that $\tcub$ is in fact a cuboctahedron in
which a triangular face is rotated by an angle of $\pi/3,$ about its center and in the
plane of the triangle.
We will also use the octahedron
\begin{align*}
 \oct:=\cfcc\cap B\left(2^{-\frac{1}{2}}, 2^{-\frac{1}{2}}(1,0,0)^T\right).
\end{align*}
\begin{prop} \label{fccchar}
Let $\c' \subset \R^3$ be a set  with the property that for each $z \in \c'$
\begin{enumerate}
\item $\left|z-z'\right|\geq 1$ for all $z'\in\c' \setminus\{z\}$.
\item There are exactly 12 points $z'\in\c'$ such
  that $\left|z-z'\right|=1.$
\item There are exactly 48 pairs $z_1,z_2 \in \c'$ such that
  $\left|z_1-z_2\right|=\left|z-z_i\right|=1$ for $i\in\{1,2\}.$
\item There are exactly 48 pairs $z_1, z_2\in\c'$ such that
 $\left|z_1-z_2\right|=\sqrt{3}$ and $|z_i - z|=1$ for $i \in \{1,2\}$.
\end{enumerate}
Then there exists a translation $t \in \R^3$ and a rotation $R \in SO(3)$ such
that $R \c' +t = \cfcc$.
\end{prop}
\begin{proof}
Properties~1-3 are
sufficient to ensure that, for every $z\in\c',$ the set $(\sx+z)\cap\c'$ is
either a rotated and translated cuboctahedron, or a twisted cuboctahedron;
this is a consequence of Theorem~\ref{thm:HTT}.
Property 4 then selects the cuboctahedron.

By induction one can see that each cuboctahedron is a translated copy of a single
rotated cuboctahedron, i.e. there exists a translation $t \in \R^3$ and a rotation
$R \in SO(3)$ such that
$$ (\sx+z) \cap \c' = R \cub +t \text{ for all } z \in \c',$$
which concludes the proof.
\end{proof}
\begin{defn}[Contact graphs]\label{defn:contact}
 For any discrete set $Z\subset \sx,$ the associated contact graph $CG(Z)$ is defined
 to be the graph with vertices at points in $Z$ and edges $\{z_1,z_2\}$ such that
 $z_1,z_2\in Z$ and $\left|z_1-z_2\right|=1.$
\end{defn}

\begin{figure}[H]\label{fig:fcchcpgraphs}
\begin{center}
\includegraphics[trim=0cm 5cm 0cm 5cm, clip, width=3.5cm]{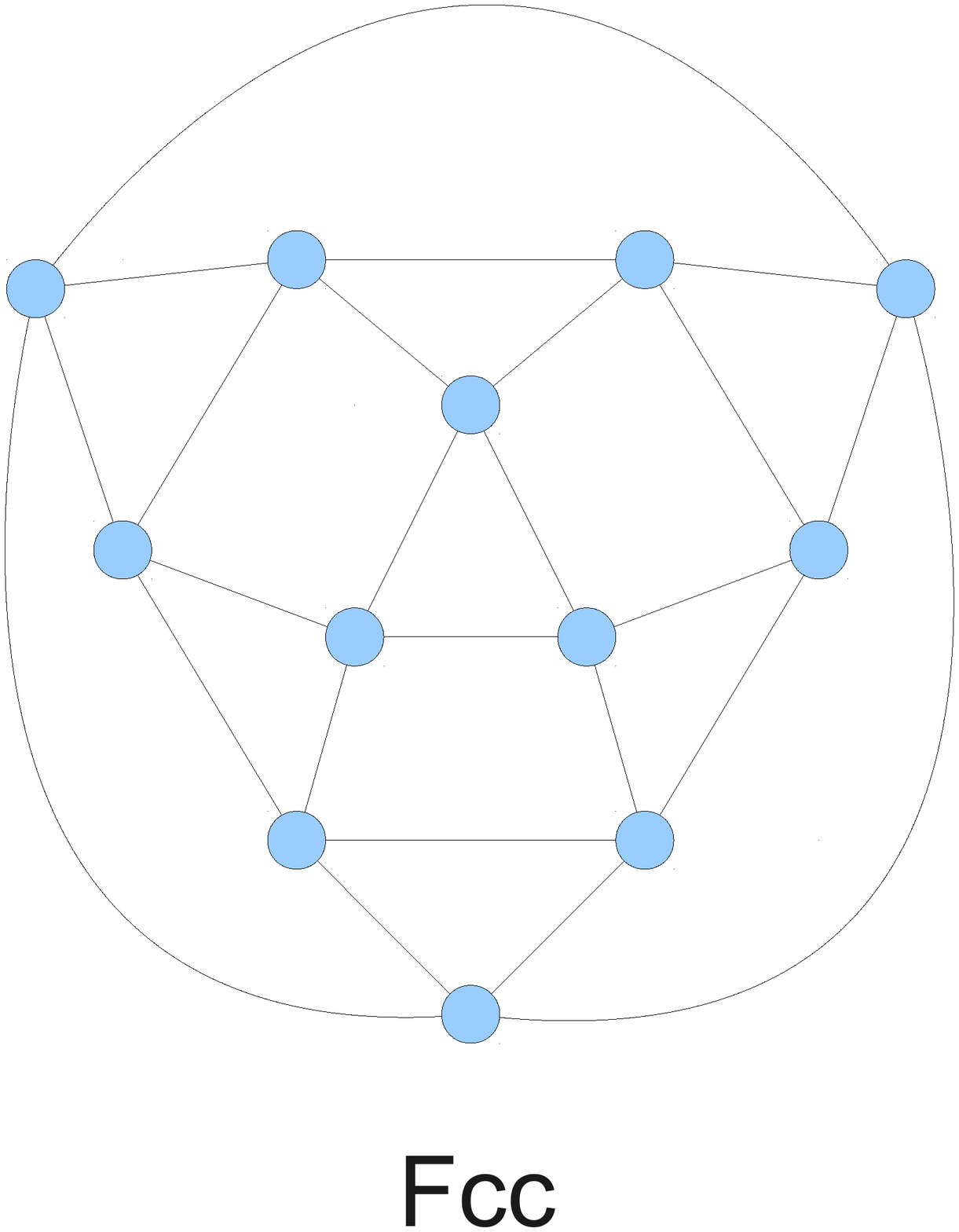} \ \ \ \ \ \ \ \ \ \ \ \ \ \ \includegraphics[trim=0cm 5cm 0cm 5cm, clip, width=3.51cm]{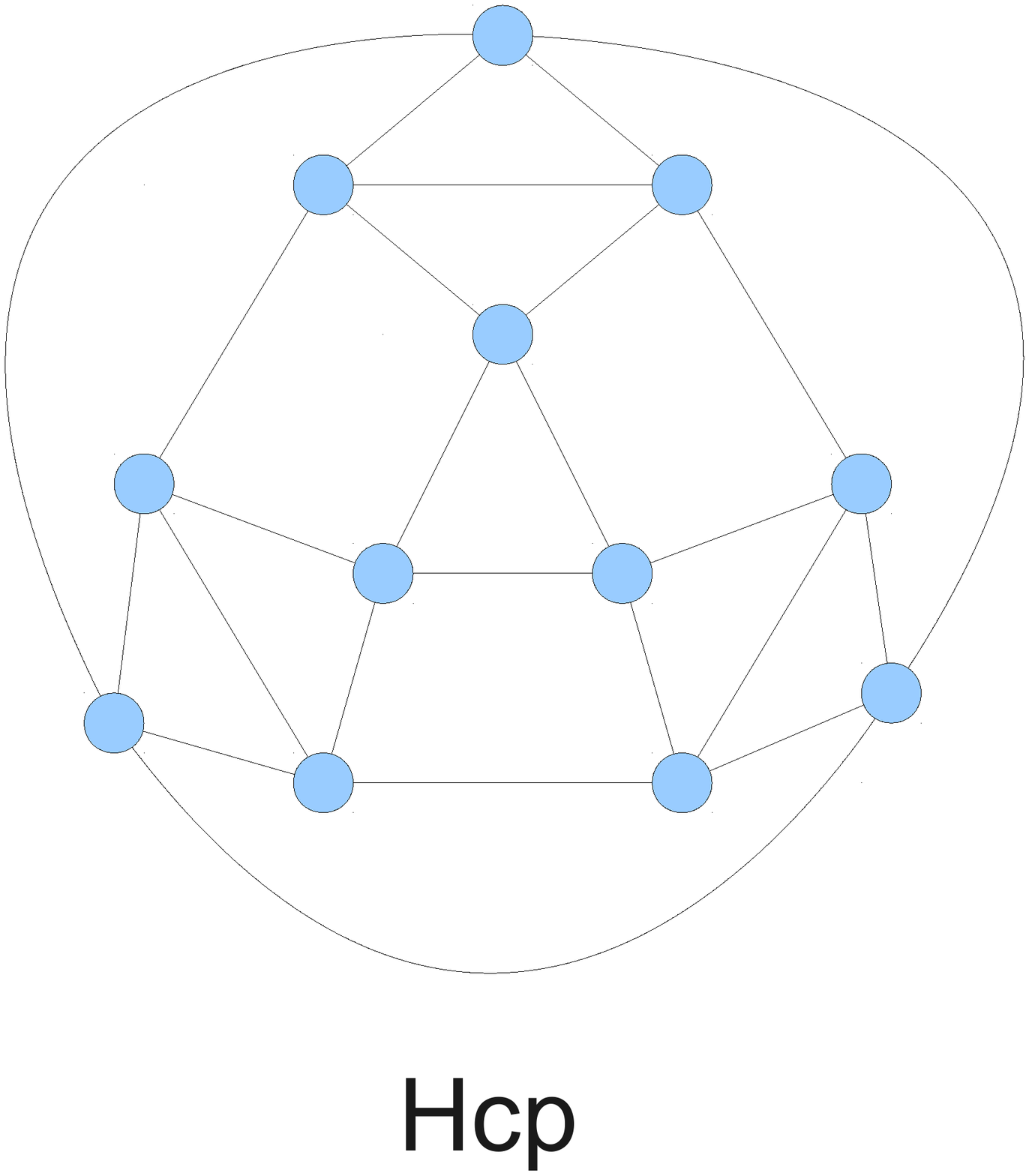}
\caption{Contact graphs of $\cub$ and $\tcub$ respectively}
\end{center}
\end{figure}

\subsubsection{Nearest neighborhood geometry of the ground state}
\label{sec:localground}
We denote by
\begin{align} \label{pairtrip}
P:=\left\{(x,x')\in X\times X :x\neq x'\right\}
\end{align}
the set of ordered pairs. We will use the convention that
if $p\in P,$ then $p=(p_1,p_2).$  For clarity, we will generally write $p\in P$ to denote a long-range pair, and $q\in \nnb$ to denote an edge (c.f. (\ref{eq:sdefn})).  For each
$x\in X,$ we define the \textit{nearest neighborhood} of $x$ by
\begin{eqnarray}\label{eq:firstnhd}
N(x):=\left\{x'\in X:(x,x')\in \nnb\right\}
\end{eqnarray}
and define $x$ and $x'$ as nearest neighbors if $x'\in N(x).$

For each $x\in X,$ we define
\begin{align*}
\angset(x):=\left\{q\in \nnb:q \subset N(x)\right\},
\end{align*}
to be the set of edges contained within the nearest neighborhood of $x$ and
\begin{align} \label{deftrip}
T:=\{(x,x',x'') \in X\times X \times X \; : \; (x', x'') \in \angset(x)\}
\end{align}
the set of  neighboring triples.

The following proposition provides upper bounds on $\# N(x)$ and
$\#\angset(x).$  If both upper bounds are attained, then statement 2 says that
$y(N(x))$ approximates a rotated and translated subset of the fcc or hcp crystal lattice.
This motivates the definition of a set of \textit{regular points} in $X$ (c.f.
Definition~\ref{defn:regular}). Note that the concept of regular points does not
discriminate between $\cfcc$ and $\chcp$.

\begin{prop}[Local neighborhoods]\label{prop:maxnhd}
There exists a constant $\alpha_0>0$  such that for all $\alpha\in (0,\alpha_0)$ and
and all configurations $y:X\rightarrow\re^3$ satisfying the minimum distance bound (\ref{eq:mindist})
the following statements are true.

\begin{enumerate}
\item $\# N(x)\leq 12$ and $\half\#\angset(x)\leq 24$ for all $x\in X.$
\item If $\half\#\angset(x)=24,$ then $\# N(x) = 12$ and there exists
$Q\in\{\cub,\tcub\},$ a map $\Phi:Q \cup\{0\}\to N(x)\cup\{x\}$ \lf{and} a monotone function $\veps:\re\rightarrow\re$ such that $\lim_{\alpha\rightarrow 0}\veps(\alpha)=0$ and
\begin{align} \label{quasirigid}
& (\Phi(\eta),\Phi(\eta')) \in \nnb \text{ if and only if } \eta,\eta' \in Q\cup\{0\} \text{ and } |\eta-\eta'| =1,\\
\label{eq:qdist}
& \min_{R \in SO(3)}\amend{9}{\max_{\eta \in Q \cup \{0\}}}|R\eta +y(x)-y\circ \Phi(\eta)|\leq\veps(\alpha).
\end{align}
\end{enumerate}
\end{prop}
Note that \eqref{eq:qdist} implies \amend{13}{$\Phi(0) =x$}.
The proof of Proposition \ref{prop:maxnhd} depends on two key results: the
three-dimensional kissing problem (Theorem~\ref{thm:kissing}) and the maximum
number of tangencies in a kissing configuration of unit spheres (Theorem~\ref{thm:HTT}).

\begin{thm}[The kissing problem]\label{thm:kissing}
For any $d\in\{1,2\ldots \}$ let the kissing number $k(d)$ be the maximum number of
non-overlapping unit spheres in $\re^d$ that can simultaneously touch a central
unit sphere.  Then $k(3)=12.$
 \end{thm}
The first proof that $k(3)=12$ was given by Sch\"utte and Van der Waerden in 1953
\cite{SchVdWae}, followed by an independent proof by Leech in 1956 \cite{Lee}.

\begin{thm}\label{thm:HTT}
Let $Z\subset \sx$ be a discrete set of vertices such that $\left|z'-z\right|\geq 1$
for all $\{z,z'\}\subset Z.$  Then the maximal number \lf{of} undirected
of edges in the contact graph $CG(Z)$ (cf Def. \ref{defn:contact}) is 24.
Equality is attained only when the points of $Z$ are placed at the vertices of a
cuboctahedron or a twisted cuboctahedron, with edges of unit length.
\end{thm}
\begin{proof}
See \cite{HarTayThe}
\end{proof}

\begin{proof}[Proof of Proposition \ref{prop:maxnhd}]
The proof of statement\lf{s} 1 and 2 \lf{is} an immediate consequence of
Theorems~\ref{thm:kissing} \& \ref{thm:HTT} and standard compactness arguments.

\end{proof}

The dichotomy result in Proposition~\ref{prop:maxnhd} allows us
to partition the label set $X$.

\begin{defn}\label{defn:regular}
The subsets $X_{12}, \Xcub, \Xtcub, \partial X$ of $X$ are defined as
\begin{align*}
X_{12} & = \{ x \in X \; : \; \#N(x) = 12\},\\
\Xreg & = \left\{ x\in X_{12} \; : \; \tfrac{1}{2}\# A(x)
= 24 \right\},\\
\Xcub &= \left\{ x \in X_\mathrm{reg} \; : \; (\ref{eq:qdist})
\text{ holds with } Q=\cub \right\},\\
\Xtcub &= \left\{ x \in X_\mathrm{reg} \; : \;
(\ref{eq:qdist}) \text{ holds with } Q=\tcub\right\},\\
\partial X &= X \setminus \Xcub.
\end{align*}
\end{defn}
Clearly $X \supset X_{12} \supset \Xreg$. Proposition~\ref{prop:maxnhd}.2 implies that $\Xcub$ and $\Xtcub$ form a partition of $\Xreg$, i.e.
\begin{align} \label{labdecomp}
\Xreg = \Xcub \cup \Xtcub \text{ and }
\Xcub \cap \Xtcub=\emptyset
\end{align}
if $\alpha \leq \alpha_0$.

If $x \in X_\mathrm{reg}$, then \amend{10}{Proposition \ref{prop:maxnhd}.2} allows us to identify subsets of $N(x)$ which form triangles and squares. As an application of this construction we can characterize the set of regular points with a complete set of second neighbors.
\begin{defn}
The regular points with complete second neighborhood are defined by
\begin{align*}
\Xreg^2 &=\left\{ x \in \Xreg \; : \; N(x) \subset \Xreg\right\}.
\end{align*}
The second neighborhood of a label $x \in \Xreg^2$ is defined by
\begin{align*}
N^2(x) &= \bigcup_{\atop{\{x_1,\ldots,x_4\} \subset N(x)}{\{x_1,\ldots,x_4\} \text{ is  a square}}}\left( \bigcap_{i=1}^4 N(x_i)\right)
\setminus\{x\},
\end{align*}
\ft{with the convention that a set $\{x_1,\ldots,x_4\}$ is called a square if it correponds to one of the six
squares in the
contact graph of $\tcub$ and $\cub$, cf. fig.~\ref{fig:fcchcpgraphs}.}
\end{defn}

\subsection{Simplicial decomposition of $\c$}
If $\c= \cfcc$ or $\c=\chcp$ it is possible to cover $\R^3$ by tetrahedra and octahedra such that the corners coincide with $\c$ and almost every point in $\R^3$ is covered exactly once. \amend{11}{To see this we recall that
$\c$ can be written as unions of layers of triangular lattices and observe that it suffices to decompose the
space between two consecutive layers such that the surface is given by two parallel planes. An illustration of the
(actually unique) decomposition is given in fig.~\ref{units-illustration}.}
\begin{figure}\begin{center}
 \includegraphics[scale=0.7]{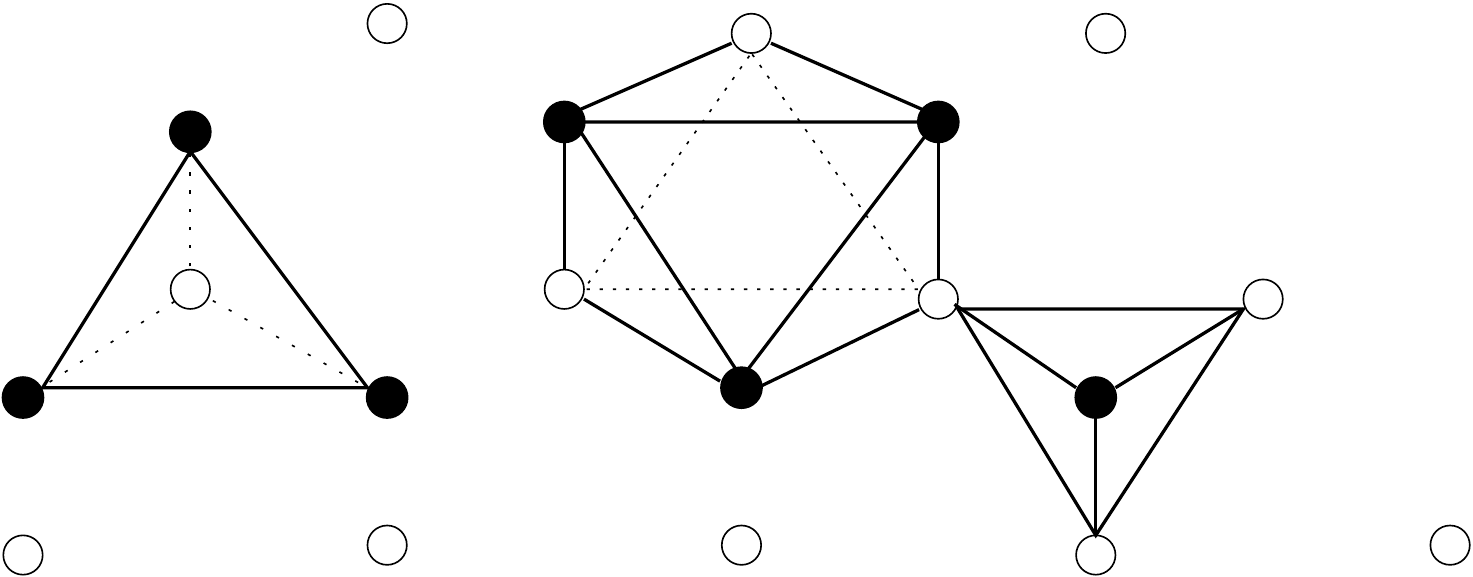}
\end{center}
\label{units-illustration}
\caption{Two consecutive layers of a close packed lattice. Solid bullets are in the upper layer, empty bullets in the
lower layer. The tetrahedra and octahedra are defined by the rule: Upper triangles with sidelength 1 that do not contain a solid bullet bound octahedra, the remaining upper triangles with sidelength 1 bound tetrahedra and conversely for lower triangles}
\end{figure}

We introduce two families of sets:
\textit{units} $\mathcal U$ and \textit{simplices} $\d$. Units are either tetrahedra or octahedra.
Each octahedron can be decomposed into 8 simplices, the tetrahedra are retained without modification. The simplices provide an intuitive notion of piecewise affine interpolation.
\begin{defn}[Units and simplices]\label{defn:simplices}
Let $\c$ be either $\cfcc$ or $\chcp$. The units are given by
$$ \mathcal U = \{\tau \subset \c \; : \; (\# \tau = 6 \text{ and } \diam(\tau)
= \sqrt{2}) \text{ or } (\#\tau = 4 \text{ and } \diam(\tau) = 1)\},$$
with $\diam(\tau) = \max\left\{|\eta-\eta'| \; : \; \eta,\eta' \in \tau\right\}$.
The centers of the octahedra are defined by
$$\c_* = \left\{ \frac{1}{6}\sum_{\eta \in \tau} \eta \; : \;
\tau \in \mathcal U \text{ and } \#\tau = 6\right\}.$$
The simplices are given by
$$ \mathcal D = \left\{\sigma \subset \c \cup \c_* \;:\; \# \sigma = 4 \text{ and }
\diam(\sigma)= 1\right\}.
$$
\end{defn}
It is easy to see that each unit is either a tetrahedron or an octahedron with sidelength 1. Moreover, both $\mathcal U$ and $\mathcal D$ form a
disjoint covering of $\R^3$, i.e.
\begin{align*}
\bigcup\limits_{\tau \in \mathcal U} \conv(\tau) = \bigcup_{\sigma \in \mathcal D} \conv(\sigma) = \R^3,
\end{align*}
and
\begin{align*}
\meas(\conv(\tau)\cap \conv(\tau'))=\meas(\conv(\sigma) \cap \conv(\sigma')) = 0,
\end{align*}
for all $\tau, \tau' \in \mathcal U$, $\sigma, \sigma' \in \mathcal D$
such that $\tau \neq \tau'$ and $\sigma \neq \sigma'$. The simplicial
decomposition of $\R^3$ is finer than the decomposition into units, i.e.
for each unit $\tau$ there are simplices $\sigma_1\ldots \sigma_I$ such that
\begin{align} \label{eq:unitdecomp}
 \conv(\tau) = \bigcup_{i=1}^I \conv(\sigma_i),
\end{align}
where $I=1$ if $\tau$ is a tetrahedron and $I=8$ if $\tau$ is an octahedron.

Recall that the contact graphs of $\cub$ and $\tcub$ contain 6 rigid squares.
\begin{remark} \label{rem:oct}
Note that for $\c = \cfcc$ a scaled octahedron $\Omega = s\, \conv(\oct)$, $s \in \{1,2,\ldots\}$ admits a decomposition into units in the sense that
$$ \Omega=\bigcup_{\tau \in \mathcal A}\conv(\tau) \text{ for some
collection of units } \mathcal A \subset \mathcal U. $$
 To see this it suffices to note that the boundary of $\Omega$ is a subset of the union of 8 triangular lattice planes which do not cut units.
\end{remark}
Now we are in a position to introduce interpolations and reference configurations.

\begin{defn}[$\Phi-$interpolation maps]\label{defn:admap}
Let $\c$ be either $\cfcc$ or $\chcp$ and $\Omega \subset \R^3$ be a simply connected set such that
\begin{align} \label{decomp}
&\amend{12}{\Omega=\cup_{\tau \in \mathcal A}\conv(\tau)} \text{ for some
collection of units } \mathcal A \subset \mathcal U. 
\end{align}
The map $u \in W^{1,\infty}(\Omega)$ is an interpolation of
$\Phi: \c \cap \Omega \to X$ if $u|_{\conv(\sigma)}$ is affine for each simplex \ft{$\sigma \in \mathcal D$ such
that $\sigma \subset \Omega$} and
$$ u(\eta) = \left\{\begin{array}{rl}
y\circ\Phi(\eta)  & \text{ if } \eta \in \c \cap \Omega,\\[0.5em]
\frac{1}{6} \sum\limits_{\atop{\eta' \in \c}{|\eta'-\eta|=\frac{1}{\ft{2}\sqrt{2}}}}u(\eta') & \text{ if } \eta \in
\c_* \cap \Omega.\end{array}\right. $$
\end{defn}

\begin{defn}[Reference configuration]\label{defn:discretelocal}
  Let $\c$ be either $\cfcc$ or $\chcp$ and $y:X \to \R^3$ be a configuration map satisfying the minimum distance
  bound (\ref{eq:mindist}) and $A \subset X$. A triple $(\Omega, \Phi,u)$ is
a reference configuration covering $A$ if $\Omega\subset\R^3$ is
simply connected, eqn.~(\ref{decomp}) holds, the map
$u \in W^{1,\infty}(\Omega,\R^3)$ is an interpolation of
$\Phi: \Omega \cap \c\to X$
in the sense of Definition~\ref{defn:admap} and
\begin{enumerate}
\item The map $\Phi$ covers $A$, i.e.
$A \subset \Phi(\Omega \cap \c)$.
\item $(\Phi(\eta),\Phi(\eta')) \in \nnb$ if and only if $|\eta -\eta'|=1$, $\eta,\eta' \in \Omega \cap \c$.
\item The map $u$ satisfies the bound
\begin{align}\label{eq:orpreserving}
\|\dist(\nabla u, SO(3))\|_{L^\infty(\Omega)} \leq \frac{1}{2}.
\end{align}
\end{enumerate}
\end{defn}
Inequality~(\ref{eq:orpreserving}) guarantees that the local
orientation is preserved by $u$.
\subsection{Existence of reference configurations} \label{sec:refconex}
We state existence results for reference configurations which cover
defect-free subsets of $X$. The construction is based on the existence of local imbeddings $\Phi$ which
map $\cub$ bijectively to neighborhoods of labels $x \in X\setminus \partial X$. These imbeddings can be chosen in a compatible way in the sense that they coincide locally after rotation and
translation.

\begin{prop}[Compatibility of local imbeddings]
\label{prop:discretelocal}
Let $x, x' \in X_\mathrm{reg}$, $Q, Q' \in \{\cub, \tcub\}$, $R, R' \in SO(3)$, $\Phi:Q\cup\{0\} \to X$, $\Phi':Q'\cup\{0\} \to X,$ \lf{$\varepsilon:\re\to\re$} be the associated domains, rotations and maps from Proposition~\ref{prop:maxnhd}. If $x' \in N(x)\cup\{x\}$ \lf{and if $\alpha>0$ is sufficiently small}, then
there exists a rotation $T \in SO(3)$ such that the set
$A=(Q \cup \{0\}) \cap (T\,(Q' \cup \{0\})+\Phi^{-1}(x'))$
has at least 6 elements and
$\Phi'$ is compatible with $\Phi$ in the sense that
\begin{align}\label{eq:nnunique}
\Phi'(T^{-1}(\eta-\xi)) =\Phi(\eta) \text{ for all } \eta \in A,
\end{align}
and
\begin{align}
 \label{eq:rotcomp} |T-R^{-1}R'|\leq 6 \,\veps\lf{(\alpha)},
\end{align}
where $\xi = \Phi^{-1}(x')-\Phi^{-1}(x)$ and $|\ft{F}|= \max_{|v|\leq 1} |\ft{F}v|$ denotes the operator norm.
\end{prop}
\begin{proof}
See appendix.
\end{proof}
If $N(x) \subset X_\mathrm{reg}$, then we can construct a reference configuration
which covers $N^2(x)\cup N(x) \cup \{x\}$.
\begin{prop}
\label{prop:locref} Let $x \in \Xreg^2$.
If $\alpha$ is sufficiently small, then there exists a reference configuration $(\Omega,\Phi,u)$ covering $N^2(x)\cup N(x) \cup\{x\}$ such that
\begin{align}
\label{eq:snd} \left|\Phi^{-1}(x)-\Phi^{-1}(x')\right|=\sqrt{2}
\end{align}
for all $x' \in N^2(x)$.
\end{prop}
\begin{proof}
See appendix.
\end{proof}
It is easy to see that the domain $\Omega$ in Proposition~\ref{prop:locref} is a regular octahedron with sidelength 2 if $Q = \cub$ \lf{in (\ref{eq:qdist}).} Large-scale imbeddings can be constructed by piecing together local imbeddings.
It will be important for the subsequent analysis that for certain
reference configurations $(\Omega,\Phi,u)$ the rigidity constant (cf. Section~\ref{sec:rigidity}) of the domain $\Omega$ is uniformly bounded.
\begin{prop}[Existence of a reference configuration]
\label{prop:refconfig}
  There exists $\alpha_0>0$ such that for all $\alpha\in(0,\alpha_0),$ \ $r\geq 1-\alpha$, $x\in X$ and $y:X\rightarrow\re^3$ satisfying the minimum distance bound (\ref{eq:mindist}) and
  $$\dist(\{y(x)\}, y(\partial X)) \geq 2\,r+3,$$
  there exists a reference configuration $(\Omega, \Phi,u)$ covering $\{x\}$ such that
$\Omega=s\,\oct$ (a scaled octahedron) with
\begin{align*}
s = \min\left\{s' \in \mathbb{Z}\; : \; s' \geq \frac{5}{2}r+3\right\},
\end{align*}
and $B(\Phi^{-1}(x),r)\subset \Omega$. Furthermore, \lf{there exists a universal constant $C>0$ such that} the interpolation $u:\Omega \to \R^3$ has the property
\begin{align} \label{eq:nablaso}
 \|\dist(\nabla u, SO(3))\|_{L^\infty(\Omega)}\leq C\alpha.
\end{align}

\end{prop}

\begin{proof}
 See appendix.
\end{proof}

\section{Rigidity bounds}\label{sec:rigidity}
The purpose of this section is to establish $L^2$ and $L^{\infty}$ rigidity estimates, which quantify the
deformations of the ground state. The bounds are based on the concept of a reference configuration.
\lf{Propositions~\ref{prop:distsum} and \ref{prop:distsum2} will imply} that $L^2$ deformations, in defect-free
regions, are controlled by a quadratic sum of edge length distortions
(c.f. (\ref{eq:ltwo})).  Our proof follows methods used previously in
\cite{Sch1},\cite{Sch2},\cite{The1} and references therein.  The $L^{\infty}$
estimate (\ref{eq:distortionineq}) is required to control distortion terms
which later arise in the Taylor expansion of the ground state energy
(c.f. (\ref{eq:deltaest})). The basic bound is provided by the following
rigidity estimate.
\begin{thm}[\cite{FriJamMul}] \label{thm:FJM} Let $d \in \{1,2, \ldots\}$, $s \in (1,\infty)$ and
$\Omega \subset \R^d$ be a simply connected Lipschitz-domain. Then there exists a constant $\lf{C}=C(\Omega,s)$ such that
\begin{align} \label{fjmineq}
  \min_{R \in SO(d)} \| \nabla u - R\|_{L^s(\Omega)} \leq C\,\|\dist(\nabla u,
  SO(d))\|_{L^s(\Omega)}
\end{align}
for all $u \in W^{1,s}(\Omega)$.  The rigidity constant $C(\Omega, s)$ is invariant under
dilations, rotations and translations, i.e.
\begin{align} \label{scalinv}
C(r\, R\, \Omega+t,s) = C(\Omega,s)
\end{align}
for all $r>0$, $R \in SO(d)$ and $t \in \R^d$.
\end{thm}
For any bounded domain $\Omega\subset\re^3,$ define
\begin{eqnarray}\label{eq:sonedomain}
 \nnb(\O):=\left\{(q_1,q_2)\in \nnb\;: \; q_1,q_2 \in \Phi(\O) \right\}
\end{eqnarray}
to be the set of edges with end-points in $\O.$
\begin{prop} \label{prop:distsum}
Let $(\Omega,\Phi,u)$ be a reference configuration such that
$\Omega$  is a scaled octahedron, i.e.
$$ \Omega = s\,\conv(\oct)$$
for some $s\in\{1,2,\ldots\}$.
Then there exists a universal constant $C>0$ such that
the map $u$ satisfies the global rigidity estimates
\begin{align}\label{eq:ltwo}
&\min_{R \in SO(3)} \|\nabla u-R\|_{L^2(\O)}^2\leq C\sum_{q\in \nnb(\Omega)}
 \left|\left|y(q_2)-y(q_1)\right|-1\right|^2,\\
 \label{eq:distortionineq}
& \left|\frac{|u(\eta)-u(\eta')|}{|\eta-\eta'|}-1\right|
\leq C\,\alpha \text{ for all } \eta,\eta' \in \Omega \text{ such that } \eta\neq \eta'.
\end{align}
\end{prop}
Note that estimate~(\ref{eq:distortionineq}) implies the injectivity of the maps $u$ and $\Phi$.
The proof of Proposition~\ref{prop:distsum} relies on
Lemma~\ref{lemma:wcell} which provides bounds for $\dist(\nabla u, SO(3))$ in terms of
$||y(x)-y(x')|-1|$, ${(x,x') \in \nnb}$.
\begin{lemma}\label{lemma:wcell}
Let $\tau \in \mathcal U$ be a unit and $u:\conv(\tau) \to \R^3$ be affine on $\conv(\sigma)$ for each simplex $\sigma \subset \conv(\tau)$ such that
$$ u\left(\frac{1}{\# \tau} \sum_{\eta \in \tau} \eta\right) =
\frac{1}{\# \tau} \sum_{\eta \in \tau} u\left(\eta\right).$$
There \amend{15}{exist} \amend{16}{universal} constants $C,c>0$ such that the function
$W_{\tau}:(\re^{3})^\tau \rightarrow [0,\infty)$ which is defined by
\begin{align*}
 W_{\tau}(u):=\sum_{\substack{\eta,\eta' \in \tau:\\
 \left|\eta-\eta'\right|= 1}}\left|\left|u(\eta)-u(\eta')\right|-1\right|^2
\end{align*}
satisfies the bound
\begin{align*}
\min_{R \in SO(3)} \|\nabla u- R\|^2_{L^2(\conv(\tau))}\leq C\,W_\tau(u)
\end{align*}
as long as $\|\dist(\nabla u, SO(3))\|_{L^\infty(\conv(\tau))} \leq \amend{16}{c}$.
\end{lemma}
The map $u$ which exchanges the positions of two neighboring (opposing) vertices if $\tau$ is an tetrahedron (octahedron) and keeps the other positions fixed has the property $W_\tau(u)=0$. Thus,
the assumption that $\|\dist(\nabla u, SO(3))\|_{L^\infty(\conv(\tau))} \leq c$ can not be dropped.
\begin{proof}
See appendix.
\end{proof}

\begin{proof}[Proof of Proposition \ref{prop:distsum}]
We first prove (\ref{eq:ltwo}). Let $\mathcal A \subset \mathcal U$ be a collection of units such that $\Omega = \cup_{\tau \in \mathcal A} \conv(\tau)$.
Thanks to the rigidity estimate~(\ref{fjmineq}) we find that
\begin{eqnarray}\label{eq:qrest}
\min_{R \in SO(3)} \|\nabla u-R\|_{L^2(\Omega)}^2\leq C \sum_{\tau \in \mathcal A}
\min_{R \in SO(3)}\|\nabla u-R\|_{L^2(\conv(\tau))}^2.
\end{eqnarray}
Lemma~\ref{lemma:wcell} implies that
$$ \min_{R \in SO(3)} \|\nabla u-R\|^2_{L^2(\conv(\tau))} \leq C \sum_{\atop{\eta, \eta' \in \tau}{|\eta-\eta'|=1}}
\left|\left|u(\eta)-u(\eta')\right|-|\eta-\eta'|\right|^2,$$
which establishes~(\ref{eq:ltwo}).

Now we prove (\ref{eq:distortionineq}).
Define $U= \conv(\oct)$.
Later we will establish the existence of a universal constant $C>0$ with the property that for each
$\eta,\eta' \in \Omega$ there exists $r>0,t\in \R^3$ such that
\begin{align} \label{eq:rbound}
r &\leq C|\eta-\eta'|,\\
\label{eq:nested}
\{\eta , \eta'\} &\subset r U + t \subset \Omega.
\end{align}

Define $v(w)= \frac{1}{r}u(r\,w+t)$ for $w \in U$.
Then
\begin{align*}
\|\dist(\nabla v, SO(3))\|_{L^4(U)} \leq C\,\gamma
\end{align*}
where $\gamma = \|\dist(\nabla u,SO(3))\|_{L^\infty(U)}$.
Theorem~\ref{thm:FJM} implies that there exists $R\in SO(3)$ such that
\begin{align*}
\|\nabla v- R\|_{L^4(U)} \leq C\,\gamma.
\end{align*}
Let $\varphi(w) = v(w)-Rw$, then
Morrey's theorem delivers the existence of $\tau$ such that the $L^\infty$-bound
\begin{align*}
\|\varphi-\tau\|_{L^\infty(U)} \leq C \gamma
\end{align*}
holds. Setting $w=\frac{1}{r}(\eta-t)$ and $w'=
\frac{1}{r} (\eta'-t)$ one obtains
\begin{align*}
&\amend{19}{\left||u(\eta)-u(\eta')| -|\eta -\eta'|\right|}
=  r\,\left||\varphi(w)-\varphi(w')+R(w-w')|
-\frac{1}{r}|\eta-\eta'|\right|\\
\leq & r\,\left(\left||R(w-w')|-\frac{1}{r}|\eta-\eta'|\right|+|\varphi(w)-\tau|
+|\varphi(w')-\tau|\right)\leq C \,\gamma\, |\eta-\eta'|.
\end{align*}
The trivial identity $|R(w-w')|=\frac{1}{r}|\eta-\eta'|$ \amend{20}{and (\ref{eq:rbound})} has been used in the final inequality. Thus we have shown that
$$  \left|\frac{|u(\eta)-u(\eta')|}{|\eta-\eta'|}-1\right|
\leq C \|\dist(\nabla u,SO(3))\|_{L^\infty(\Omega)}.$$
Estimate~(\ref{eq:distortionineq}) follows now from (\ref{eq:orpreserving}) and Lemma~\ref{lemma:wcell} which implies
$$\|\dist(\nabla u,SO(3))\|_{L^\infty(\Omega)} \leq C\,\alpha.$$

Finally we prove (\ref{eq:rbound}).
Since $U$ is a convex polyhedron one obtains
$$ U = \{z \; : \; \nu_i\cdot z \leq \lambda_i, \; i=1\ldots \amend{22}{8}\} $$
for a suitable choice of
$\nu_i \in \R^3$, $|\nu_i|=1$ and $\lambda_i \in [0,\infty)$.
The constraints
(\ref{eq:nested}) imply that the optimal parameter $r$
is a solution of the linear program
\begin{align} \label{eq:linprog}
r_{\min} = \min\left\{
r \geq 0\; : \; \max\{\eta\cdot \nu_i,\eta'\cdot \nu_i\} \leq r
\lambda_i + t\cdot \nu_i \leq \lambda_i,\quad i = 1\ldots \amend{23}{8}
\right\}.
\end{align}
Note that $r_{\min}\leq 1$ since $r=1, t=0$ is admissible.
Assume that there exists sequences $\eta_n \neq \eta_n' \in U$ such that
\begin{align} \label{eq:contra}
\lim_{n \to \infty} \frac{r_{\min}(\eta_n,\eta_n')}{|\eta_n-\eta_n'|} =
\infty.
\end{align}
We can assume without loss of generality that
\begin{align} \label{eq:sdfiff}
\limsup_{n\to \infty} \frac{\dist(\{\eta_n,\eta_n'\}, \partial U)}
{|\eta_n-\eta_n'|}<\infty.
\end{align}
Indeed, if (\ref{eq:sdfiff}) fails, then we
extract the corresponding subsequence
(not relabeled) and define
$r_n = \frac{|\eta_n-\eta_n'|}{\rho}$ and $t_n = \frac{1}{2}(\eta_n+\eta_n')-\amend{24}{r_n}\,z$,
where $\rho>0$ and
$z \in U$ have the property that $B(z,\rho) \subset U$.
Clearly $$\{\eta,\eta'\} \subset r_n U + t_n \subset \Omega$$
holds for all $n$ which are sufficiently large.

Since $U$ is a polyhedron $\partial U$ can be decomposed into 3 disjoint components:
$$ \partial U= \partial U_0 \cup \partial U_1 \cup \partial U_2,$$
which correspond to corners, edges and faces.
Let $i =0$. If there exists a subsequence (not relabeled) along which
\begin{align} \label{eq:contra1}
\limsup_{n \to \infty} \frac{\dist(\{\eta_n,\eta_n'\}, \partial U_i)}{|\eta_n-\eta_n'|} = \infty
\end{align}
does not hold, then there exists $C>0$ such that
\begin{align} \label{eq:con1}
\limsup_{n \to \infty} \frac{\max(|\eta_n-z_n|,|\eta_n'-z_n|)}{|\eta_n-\eta_n'|} \leq C,
\end{align}
where $z_n \in \partial U_i$ is the minimizer of
$d_n(z) = |z-\frac{1}{2}(\eta_n+\eta_n')|$.
After translation we can assume that $z_n =0$, define $t_n=0$ and
$$r_n = \min\{ r \; : \; \eta_n,\eta_n' \in r U\}.$$
Inequality (\ref{eq:con1}) implies that
\begin{align} \label{eq:rnbound}
\limsup_{n \to \infty} \frac{r_n}{|\eta_n-\eta_n'|} \leq C
\end{align} which is the
desired contradiction and we conclude that \amend{25}{(\ref{eq:contra1})} holds for $i=0$.

Inductively we repeat this step for $i =1,2$ and observe that (\ref{eq:contra1})
for $i-1$ implies (\ref{eq:rnbound}). Once (\ref{eq:contra1}) has been established for each $i \in \{0,1,2\}$ we have derived a contradiction
to (\ref{eq:sdfiff}) and consequently (\ref{eq:rbound}) holds.
\end{proof}
We also require the following technical result.
\begin{prop}\label{prop:distsum2}
For each $v \in \R^3\setminus\{0\}$ the inequality
\begin{align}
\min_{R\in SO(3)}\left\{\left|F-R\right|^2\;:\;\frac{F\,v}{|F\,v|}
=\frac{R\,v}{|v|}\right\}
\leq C\,\dist^2(F,SO(3))\label{eq:ltwoq}
\end{align}
holds for all $F \in \R^{3 \times 3}$ \amend{14}{such that $Fv\neq 0$}.
\end{prop}

\begin{proof}
Assume without loss of generality that $|v|=1$ and let
$S \in SO(3)$ and $U\in \R^{3 \times 3}_\mathrm{sym}$ be the polar
decomposition of $F$ i.e. $F=SU$. Let $T \in SO(3)$ be a
rotation which satisfies $T Sv = \frac{1}{|Fv|} Fv$ and leaves
the span of $\{Sv,Fv\}$ invariant. Then, $G:=TS\in SO(3)$
satisfies the constraints of the left hand side of (\ref{eq:ltwoq}).

If $\theta\in[0,2\pi]$ is the angle of rotation of $T$, then the cosine rule gives
\begin{align*}
 \amend{17}{1-\cos(\theta)= \frac{1}{2}\left|Sv-\frac{Fv}{|Fv|}\right|^2\leq
2\Big(|Sv-Fv|^2+\left|Fv-\frac{Fv}{|Fv|}\right|^2\Big)=2\Big(|Sv-Fv|^2+\left||Fv|-1\right|\Big)^2.}\end{align*}
The identity $\sin^2(\theta)=(1+\cos(\theta))(1-\cos(\theta))$ implies the bound
\begin{eqnarray*}
 \sin^2(\theta)\leq 4\left(|Fv|-1|^2+|S-F|^2\right).
\end{eqnarray*}
Moreover, one obtains
\begin{eqnarray}\label{eq:TFbound}
 |T-\I|^2= 2\left(|\cos(\theta)-1|^2+|\sin(\theta)|^2\right)
\leq 16\,(|S-F|^2+||Fv|-1|^2),
\end{eqnarray}
where we used $(1-\cos(\theta))^2 \leq 2(1-\cos(\theta))$.
This implies
\begin{eqnarray*}
 |F-G|^2\leq2\left(|F-S|^2+|T-\I|^2\right)\leq C\dist^2(F,SO(3)).
\end{eqnarray*}
\end{proof}

\section{Partitioning of the energy}\label{sec:pairs}
\subsection{Reference path sets and label path sets}
For pairs $(x,x') \in X \times X$ we wish to express global differences
$y(x)-y(x')$ in a way which recognizes the local structure of the
configuration: if the map $\gamma : \{0,\ldots,\nu\} \to X$ has the properties
$\gamma(0) = x'$, $\gamma(\nu)=x$, then
$$ y(x)-y(x') = \sum_{i=1}^\nu(y(\gamma(i))-y(\gamma(i-1)))$$
holds. This formula suggests that the sum over all pairs can be written
as the sum over all such maps which will be denoted as paths from now on.
To formalize this concept we have to introduce
some structure to avoid double counting.

We denote by $\b$ the set of \textnormal{ordered bases} of \lf{$\cfcc:$}
\begin{align}\label{eq:basisdefn}
 \b:=\left\{B \in \R^{3\times 3}\; : \; \det(B)\neq 0
\text{ and } B e_i \in\lf{\cfcc}, |B e_i|=1,\, i = 1,2,3\right\},
\end{align}
with the standard convention $e_1=(1,0,0)^T$ etc. By abuse of notation we
write $v\in B$ if $v= B e_i$ for some $i \in \{1,2,3\}$.

\amend{27}{We denote by $\Lambda$ the set of all positive fcc lattice distances and define the
medium and long distances
\begin{align} \label{lamblist}
 &\Lambda = \left\{\left|z\right|\;:\;z\in\lf{\cfcc}\right\}=\{1,\sqrt{2},\sqrt{3},\sqrt{11/3},\ldots\},\\
\nonumber
 &\Lambda_\mathrm{med}=\left\{\sqrt{2},\sqrt{8/3}, \sqrt{3}\right\},\; \Lambda_\mathrm{long}=\Lambda \cap \left(\sqrt{3},\infty\right).
\end{align}
Note that $\Lambda_\mathrm{long} \subset \Lambda$, but $\Lambda_\mathrm{med}\setminus \Lambda = \{ \sqrt{8/3}\}$. The additional distance is included to facilitate the quantification of energy contributions created by the parts of the configuration with hcp structure.}

\begin{defn}[Admissible paths]\label{defn:refpaths}
  For given $\amend{26}{\nu \in \n}$ and $B \in \b$ we say that a finite sequence
  $\mu(i) \in \lf{\cfcc}$, $i=0\ldots \nu$ is an admissible (reference) path if
\begin{align}
\label{eq:gammafccdefn}
 \mu(j)-\mu(j-1)= B e_{i_j}
\end{align}
for some monotonic sequence $i_j\in \{1,2,3\}$ i.e.
$\mu$ consists of maximally three straight segments with directions given by the columns of $B$. We denote by $\Gamma[B]$ the set of such
paths, and define
\begin{align*}
k(\mu) & = \mu(\nu)-\mu(0),\\
 \Gamma(\lambda) &= \left\{\mu \in \cup_{B \in \b} \Gamma[B]\; :
\; |k(\mu)| = \lambda\right\},
\end{align*}
if $\lambda >\sqrt{3}$.
The set of paths with medium length is defined by
$$ \Gamma(\lambda) = \left\{\mu:\{0,1,2\} \to \lf{\cfcc}\; : \; |\mu(1)-\mu(0)|=|\mu(2)-\mu(1)|=1
\text{ and } |\mu(2)-\mu(0)| = \lambda \right\},$$
if $\lambda \in \{\sqrt{2},\sqrt{3}\}$.
The set of admissible paths is defined as
$$ \Gamma = \cup_{\lambda \in \Lambda} \Gamma(\lambda).$$
\end{defn}
By abuse of notation, we will abbreviate
$v \in \{\mu(1)-\mu(0), \mu(2)-\mu(1),\ldots\}$ with $v \in \mu$.

Although the definition of the set of admissible paths is very restrictive there are many paths which connect two lattice points. This observation motivates the introduction of the number
\begin{align}\label{eq:mdefn}
M(\mu):= \frac{1}{120}\#\left\{B \in \mathcal B \;:\; \mu \in \Gamma[B]\right\} \in [0,1],
\end{align}
which has the property that for $k \in \lf{\cfcc}\setminus \{0\}$
\begin{align} \label{normalization}
\sum_{\mu(0) = 0, \; \mu(\nu) = k} M(\mu)=1.
\end{align}
Equation~(\ref{normalization}) holds because $\# \mathcal B = 2^3 * \frac{6!}{3!}$ and for a generic
lattice point there are $2^{-3} \#\mathcal B=120$ choices of $B$ such that the set
$$ \mathcal A = \{\mu \in \Gamma[B] \; : \mu(0) = 0, \mu(\nu) = k\}$$
is nonempty. The correction factor accounts for the cases where $k$ is degenerate
in the sense that $\#\mathcal A>1$.

The path sets $\Gamma(\lambda)$ inherit the symmetry properties of the
fcc lattice. To construct a suitable representation of those path-symmetries
we define for each path $\mu$ the point $\zeta(\mu)$ in the
smallest subspace containing $\mu$ by the requirement
that the end points of the line segments which constitute the affine
interpolation $\hat \mu$ of $\mu$ all have the same distance $\rho(\mu)$ from $\zeta$.

\begin{lemma} \label{lem:zetadist}
For all $\mu \in \Gamma(\lambda)$
the estimate
\begin{align} \label{eq:circrad}
 \rho(\mu) = \max_{i = 1\ldots m} |\zeta(\mu) - \mu(i)| < 2 \lambda
\end{align}
holds.
\end{lemma}
\begin{proof}
The proof is a simple exercise which is included for the convenience of the
reader.

Define the difference vectors
\begin{align*}
 v_i = e_i \cdot B^{-1} k(\mu)\,Be_i \in \R^3.
\end{align*}
Thanks to translation invariance we can assume that $\mu(0)+v_1+ \frac{1}{2}v_2=0$, i.e the mid point of the second line segment is located in the origin. The point $\zeta$ is given by the formula
$$ \zeta = \frac{1}{2}B^{-T} M B^{-1}k$$
where $M \in \R^{3 \times 3}$ is the diagonal matrix
\begin{align}
M= \mathrm{diag}(e_1\cdot B^T B(e_1+e_2),0,-e_3 \cdot B^T B(e_2 +e_3)).
\end{align}
Since $|B^{-1}| = \sqrt{2}$ and $|M| = \frac{3}{2}$ we find that $|\rho(\mu)|
\leq \frac{3}{2} \lambda$. Furthermore, as $|v_2| \leq \sqrt{2} \lambda$ and $v_2 \cdot
\zeta =0$ one obtains that
$$ \rho(\mu)\leq \frac{1}{2} \sqrt{11}.$$
\end{proof}

For a unit lattice vector $v \in \sx \cap \lf{\cfcc}$ the piecewise affine interpolant
$\hat \mu$ contains at most one line-segment which is parallel to $v$.
Let $\eta\in \R^3$ be the mid-point of that line segment and define the
affine map
$\kappa_v: \R^3\to \R^3$ by the equation
\begin{align*}
\kappa_v(y)=\eta+(\I-2v\otimes v)(y-\eta)\in\re^3.
\end{align*}
It is an easy exercise to check that $\kappa_v$ leaves
$\cfcc$ invariant and $\zeta$ is a fixed point of $\kappa_v$. Now we extend
the operation of $\kappa_v$ to the set of paths by
$$ \amend{28}{\kappa_v(\mu)(i) := \kappa_v (\mu(\nu-i)).}$$
If $v \not\in \mu$, i.e. $\mu(i+1)-\mu(i) \neq v$ for all $i$, then
$\kappa_v$ is defined as the identity,
$$ \kappa_v(\mu)(i) = \mu(i),$$
cf. Fig.~\ref{fig:reflection}.
\begin{figure}\begin{center}
\includegraphics[scale=0.5]{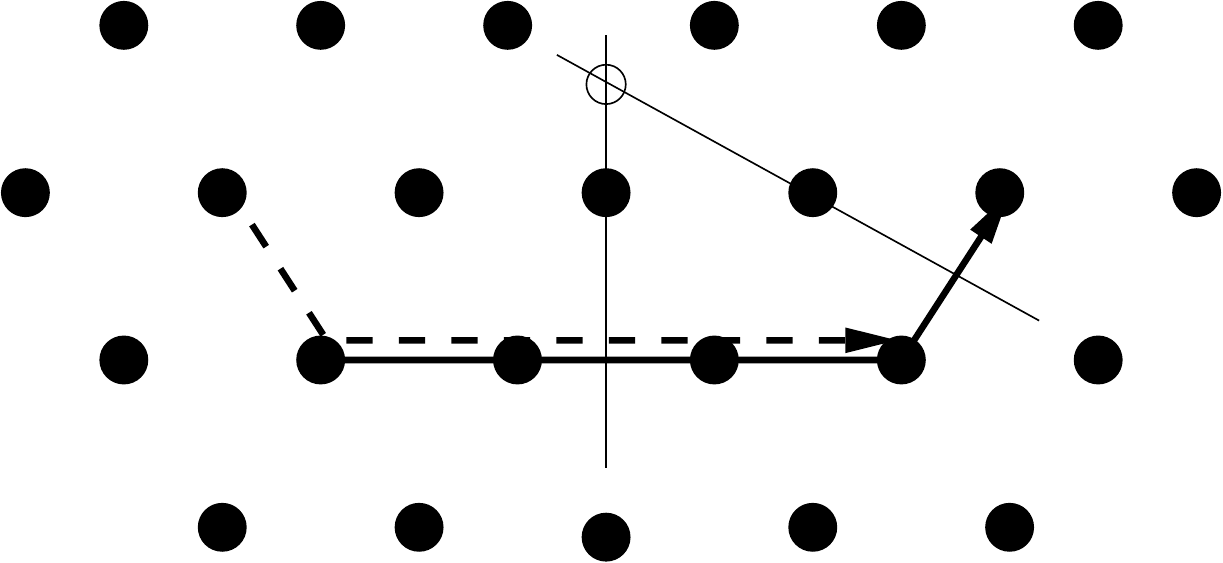}
\end{center}
 \caption{Illustration of the reflected paths. A path $\mu$ and a reflection $\kappa(\mu)$
are depicted with a solid and a dashed line. The position of the point $\zeta(\mu)$ is
indicated by an empty circle.}
\label{fig:reflection}
\end{figure}

It is an easy exercise to see that $\kappa_v:\Gamma \to \Gamma$ leaves the sets
$\Gamma(\lambda)$ invariant and has the following properties:
\begin{align}
&\kappa_v\circ \kappa_v= \Id, \label{refreflection}\\
&M(\kappa_v(\mu))=M(\mu), \label{refMinv}\\
& \label{zetainv}
\zeta(\kappa_v(\mu)) = \zeta(\mu).
\end{align}
Notice that that $\frac{1}{2}\left(\I-(\I-2v\otimes v)\right)= v\otimes v$ is
just the projection onto the span of  $\{v\}$ if $v\in \mu$. This implies that
\begin{align}
k(\mu) + k(\kappa_v(\mu))) \text{ and } v \text{ are parallel.} \label{refparallel}
\end{align}
The second key property of the maps $\kappa_v$ is that $\zeta(\mu)$ is a fixed point and the distance from $\zeta(\mu)$ is unchanged, i.e.
\begin{align} \label{circsph}
 \max_{i}|\kappa_v(\mu)(i)-\zeta(\mu)| = \max_{i}|\mu(i)-\zeta(\mu)|,
\end{align}
thus the orbit
$$ \mathcal{O}(\mu) = \bigcup_{i=1}^\infty \;\bigcup_{v_1\ldots v_i \in \sx \cap \lf{\cfcc}}
\kappa_{v_i}\circ\ldots\circ \kappa_{v_1}(\mu) \subset \Gamma $$
is a finite set.

For any $\lambda\in\Lambda,$ let
\begin{eqnarray}\label{eq:mldefn}
 m(\lambda)=\#\left\{\eta\in\lf{\cfcc}:\left|\eta\right|=\lambda\right\}
\end{eqnarray}
be the number of lattice points at distance $\lambda$ from the origin.
\begin{defn}[Label paths] \label{def:labpath}
  If $\lambda\in\Lambda$ and $\gamma:\{0,\ldots, \nu\} \to X$ is a label
  path we say that $\gamma \in \hat \Gamma_*(\lambda)$ if
there exists a reference  configuration $(\Omega,\Phi,u)$
and $\mu \in \Gamma(\lambda)$ such that
$\Omega$ is a scaled octahedron, i.e. $\Omega = s \oct$ for
some $s \in \{1,2\ldots\}$,
$\gamma = \Phi\circ \mu$ and
\begin{align} \label{eq:ballref}
&B(\zeta(\mu),2\lambda)\subset\Omega.
\end{align}
For each $\gamma \in \hat \Gamma_*(\lambda)$ we define
\begin{align*}
\hat k(\gamma) &= k(\mu),\\
\hat M(\gamma) &= M(\mu).
\end{align*}
\end{defn}
\amend{41}{We will abbreviate $q \in \{\gamma(1)-\gamma(0), \gamma(2)-\gamma(1),\ldots\}$ with $q \in \gamma$.}

For convenience, we assume for now that each label path $\gamma \in \hat \Gamma_*$
is associated with a unique reference configuration and hence drop
the notational dependency of $\hat k(\gamma)$ on
$\Phi$. The choice of the reference configuration will be specified
later.
\begin{remark} \label{fixeddomain}
Note that thanks to Lemma~\ref{lem:zetadist} and eqn.~(\ref{eq:ballref}) the orbit of
$\mu$ is contained in $\Omega$:
\begin{align*}
\mathcal{O}(\mu) \subset \Omega.
\end{align*}
Since $\Omega$ is a scaled octahedron and eqn.~(\ref{scalinv}) implies that the rigidity constant does not depend on $s$ there exists a universal constant $K>0$ such that equations~(\ref{eq:ltwo}, \ref{eq:distortionineq}) are satisfied with $C=K$.
\end{remark}
A key property of the label paths is the invariance under the the action of the reflection map $\hat \kappa$.
\begin{lemma}
Let $\lambda \in \Lambda$ and $\gamma \in \hat\Gamma_*(\lambda)$ be a label path. For each $q \in \nnb$ the reflected reference path $\kappa_{\hat k(q)}(\mu)$ is contained in $\Omega$. Moreover, the reflected label path, which is defined by
\begin{align*}
\hat \kappa_q(\gamma) = \Phi\circ\kappa_{\hat k(q)}(\mu),
\end{align*}
is an element of $\hat \Gamma_*(\lambda)$ with the same reference configuration
$(\Omega,\Phi,\lf{u})$. The reflection map
$\hat \kappa_q:\hat \Gamma_* \to \hat \Gamma_*$ has the properties
\begin{align}
& \hat\kappa_q^2= \Id, \label{reflection}\\
& \hat M(\hat \kappa_q(\gamma))= \hat M(\gamma), \label{Minv}\\
& |\hat k(\gamma)|= |\hat k(\hat \kappa_q(\gamma))|, \label{kinv}\\
& \hat k(\gamma) + \hat k(\hat \kappa_q(\gamma))\; \| \; \hat k(q) \text{ if } q \in \gamma. \label{parallel}
\end{align}
\end{lemma}
\begin{proof}
Lemma~\ref{lem:zetadist} and eqn.~(\ref{eq:ballref}) implies that $\kappa_{\hat k(q)}(\mu)$ is contained in $\Omega$.
The claim is an immediate consequence of equations~(\ref{refreflection}),
(\ref{refMinv}), (\ref{zetainv}) and (\ref{refparallel}).
\end{proof}

The concept of label paths allows us to partition the set of pairs
into disjoint subsets according to the distance.
\begin{defn}[Regular pairs and regular label paths]\label{defn:kbonds}
For $\lambda \in \Lambda$ the set of \amend{29}{regular} pairs with distance
$\lambda$ is given by
\begin{align*}
 P_*(\lambda) =
\biggl\{ p \in P \; : \; &\exists \gamma \in \hat \Gamma_*(\lambda) \; : \;
p = (\gamma(0),\gamma(\nu))
\text{ and } \dist(\{y(p_1),y(p_2)\}, y(\partial X)) \geq 10 \lambda\biggr\}
\end{align*}
and $P_*(1)=\nnb$.

The extended set of pairs is $P(\lambda) = P_*(\lambda)$ if $\lambda \in \Lambda \setminus
\Lambda_\mathrm{med}$ and
\begin{align} \label{medrpair}
 P(\lambda) =\biggl\{ p \in P \; : \;&
 \exists x \in X^2_\mathrm{reg} \text{ such that } p_1,p_2 \in N(x) \text{ and }
|\Phi^{-1}(p_2)-\Phi^{-1}(p_1)| = \lambda \biggr\}
\end{align}
if $\lambda \in \Lambda_\mathrm{med}$.

For a regular pair $p \in P_*(\lambda)$ we define $\lambda(p) = \lambda$.

The set of defect pairs is defined as
$$P_0:=P\backslash\cup_{\lambda\in\Lambda}P(\lambda).$$
The regular label paths are defined by the requirement that the end points
form regular pairs:
$$ \hat \Gamma(\lambda) = \{\gamma \in \hat \Gamma_*(\lambda) \; : \; (\gamma(0),\gamma(\nu)) \in P_*(\lambda)\},$$
$\lambda \in \Lambda$.

\lf{For a regular label path $\gamma \in \hat\Gamma_*(\lambda)$ we define $\lambda(\gamma) = \lambda$.}
\end{defn}
Equation (\ref{medrpair}) allows us to quantify the surplus energy generated by
the parts of the configuration with hcp structure, cf dichotomy~(\ref{labdecomp}).
Note that only $\hat \Gamma_*(\lambda)$, but not $\hat \Gamma(\lambda)$ is invariant under the action of the reflection map $\hat \kappa$.
\subsection{Cardinality of the sets $P(\lambda)$}
The analysis of the energy of a configuration requires a quantitative link between the label sets of Definition~\ref{defn:regular} and the pair sets $P(\lambda)$. This link is provided by Lemma~\ref{lemma:cardinality}.
\begin{prop} \label{Mintrin}
\begin{enumerate}
The pairs $P(\lambda)$ and $P_*(\lambda)$ have the following properties:
\item \label{setdis}
 $P_*(\lambda_1) \cap P_*(\lambda_2)=\emptyset$ if $\lambda_1,\lambda_2 \in \Lambda$ and
$P(\lambda_1)\cap P(\lambda_2) = \emptyset$ if $\amend{30}{\lambda_1,\lambda_2\in\Lambda_\mathrm{med}}$, $\lambda_1 \neq \lambda_2$.
\item \label{defex}
If $p \in P_*(\lambda)$, then
\begin{align} \label{pathnumb}
 \sum_{\atop{\gamma \in \hat \Gamma(\lambda)}{(\gamma(0),\gamma(\nu)) = p}}\hat M(\gamma) = 1.
\end{align}
\end{enumerate}
\end{prop}

\begin{proof}
Assume first that $p \in P(\lambda)$ for some $\lambda \in \Lambda_\mathrm{med}$.
Proposition~\ref{prop:maxnhd} implies that
\begin{align}
\label{lamcomp}
||y(p_1)-y(p_2)| -\lambda| \leq \eps\lf(\alpha).
\end{align}
For sufficiently small $\lf{\alpha}$ there is just one $\lambda \in \Lambda_\mathrm{med}$ which
satisfies (\ref{lamcomp}) because $\Lambda_\mathrm{med}$ is a finite set.

Assume next that $\lambda_1\leq \lambda_2$, $p \in P_*(\lambda_1)\cap P_*(\lambda_2)$ and let $(\Omega_i,\Phi_i,u_i)$, $i=1,2$ be the associated reference configurations.
Estimate (\ref{eq:distortionineq}) implies that $\lambda_2-\lambda_1 \leq C \alpha$.
After translation we can assume that
$\Phi_i(0) = p_1$. Proposition~\ref{prop:refconfig} implies that there exists a translated reference configuration $(\Omega_0,\Phi_0,u_0)$ such that $p_1 = \Phi_0(0)$ and $B(0,2\lambda_2)\subset \Omega_0$.
Proposition~\ref{prop:discretelocal} implies the existence of rotations $T_i$
with the property
$u_i(\eta) = u_0(T_i \eta)$ for all $\eta \in B(0,1)$, $i=1,2$.
Inductively one obtains that
\begin{align} \label{mapseq}
u_0(\eta)= u_i(T_i\eta) \text{ for all } |\eta| \leq 2\lambda_2, i=1,2.
\end{align}
The rigidity bound (\ref{eq:distortionineq}) implies that
$|\Phi_i^{-1}(p_2)| \leq \frac{3}{2}\,\lambda_2$ and together with (\ref{mapseq})
one finds
\begin{align*}
 \lambda_1= |\Phi_1^{-1}(p_2)| = |\Phi_2^{-1}(p_2)|= \lambda_2,
\end{align*}
which is claim~\ref{setdis}.

Proof of claim \ref{defex}.
Let $(\Omega_0,\Phi_0,u_0)$ the the reference configuration which is defined above and
let $\mu \in \Gamma(\lambda)$ be a reference path such that $\mu(0)=0$, $\mu(\nu)=\eta$.
Then $\gamma = \Phi_0\circ \mu \in \hat \Gamma(\lambda)$ and $(\gamma(0),\gamma(\nu))=p$,
this implies the lower bound
\begin{align*}
\sum_{\atop{\gamma \in \hat \Gamma(\lambda)}{(\gamma(0),\gamma(\nu))=p}} \hat M(\gamma) \geq
\sum_{\atop{\mu \in \Gamma(\lambda)}{(\mu(0),\mu(\nu))=\eta}} M(\mu)\stackrel{(\ref{normalization})}{=} 1.
\end{align*}
The first inequality holds because there might be label paths that require a different
reference configuration.
To prove the corresponding upper bound we assume that
$(\Omega,\Phi,u)$ is another reference configuration and
$\mu$ is a reference path such that $(\Phi\circ \mu(0),\Phi\circ \mu(\nu))=p$.
Using the same argument as in the proof of claim~\ref{setdis} one can rotate and
translate $(\Omega,\Phi,u)$. This induces a one-to-one relationship between the
label paths induced by $(\Omega,\Phi,u)$ and $(\Omega_0,\Phi_0,u_0)$ and thereby
claim~\ref{defex}.
\end{proof}

Finally we provide a quantitative link between path sets and
the label sets introduced in Definition~\ref{defn:regular}.
\begin{lemma}\label{lemma:cardinality}
There exists constants $C,\alpha_0\geq 0$ such that for every
$\alpha\in(0,\alpha_0),$ every configuration $y:X\rightarrow\re^3$
satisfying the minimum distance bound (\ref{eq:mindist}) and every
$\lambda\in\Lambda_\mathrm{long},$ the estimate
\begin{align}\label{eq:cardinality}
0\leq\frac{m(\lambda)}{m(1)}\# \nnb -\# P(\lambda)
\leq C\lambda^3\#\partial X
\end{align}
holds. Let $n=\#X$. Then the short- and medium-range pairs
satisfy the bounds
\begin{align}
& (n-\# X_{12}) \leq m(1)\,n -\# \nnb \leq m(1)(n-\# X_{12}),\label{nnbound} \\
&\lf{\#(\Xreg \setminus \Xreg^2)} \leq 12 (n-\#\Xreg), \label{eq:X2}\\
&\left|m(\sqrt{2})\,n -\# P(\sqrt{2})\right| \leq C(n - \# \Xreg),
\label{snbound}\\
&\left|\# P(\sqrt{8/3}) -2 \# \Xtcub^2\right|\leq C(n - \# \Xreg),
\label{s8bound}\\
&0 \leq
\# P(\sqrt{3})-  m(\sqrt{3})\, \# \Xcub^2 - 18 \# \Xtcub^2
\leq C(n- \# \Xreg). \label{s3bound}
\end{align}
\end{lemma}
The proof can be found in the appendix.
\section{Proof of Theorem~\ref{thm:mainintro}}
\label{sec:energy}
\begin{proof}
We adopt the notation (\ref{eq:vectnot}).
Following the remarks in Section~\ref{thm:mainintro} it suffices to establish the lower bound (\ref{lb}) which is a consequence of (\ref{finebound}).

Recall the definitions of $P(\lambda)$ and $P_*(\lambda)$ in Def.~\ref{defn:kbonds}.
The interaction energy
(\ref{eq:energy}) is written as the sum of structural, elastic and defect contributions:
\begin{align*}
  E(y)= E_\mathrm{struct}(y)+ E_\mathrm{elast}(y)  + E_\mathrm{defect}(y)
\end{align*}
where
\begin{align*}
E_\mathrm{struct}(y) &= \sum_{\lambda \in \Lambda} \sum_{p \in P(\lambda)} V(\lambda)+ \sum_{x \in X} e_3(x),\\
E_\mathrm{elast}(y) &= \sum_{\lambda \in \Lambda} \sum_{p \in P(\lambda)}
(V(|y(p_1)-y(p_2)|)-V(\lambda)),\\
E_\mathrm{defect}(y) &=\sum_{p\in P_0}V(|y(p_1)-y(p_2)|),
\end{align*}
and the associated three-body energy $e_3(x)$ is
defined by
\begin{eqnarray}\label{eq:e3defn}
  e_3(x):=2\sum_{\{x_1,x_2\}\subset  X\setminus\{x\}}V_3(y(x),y(x_1),y(x_2)).
\end{eqnarray}
\amend{31}{Let $n=\# X.$}  The aim of this section is to prove that the defect-free energy can be bounded
from below by the reference energy, together with contributions from bond
distortions and surface terms:
\begin{align}
\label{structbd} E_\mathrm{struct}(y) &\geq n \,\estar + c\, (n- \# X_\mathrm{reg})+ c\,\alpha^\frac{1}{2} \# \partial X, \\
\label{defectbd} E_\mathrm{defect}(y) & \geq C\,\alpha^\frac{1}{4}(\# X_\mathrm{reg}-n) - C \alpha\#\partial X,\\
\label{elastbd} E_\mathrm{elast}(y) &
\geq c\,\sum_{q\in \nnb}\left|\left|y(q_2)-y(q_1)\right|-1\right|^2
-C \alpha \#\partial X.
\end{align}
Estimates (\ref{structbd}), (\ref{defectbd}) and (\ref{elastbd}) together deliver the lower bound
\begin{align*}
E(y) \geq n \,\estar + c\, \alpha^\frac{1}{2}\# \partial X + c\,\sum_{q\in \nnb}
\left|\left|y(q_2)-y(q_1)\right|-1\right|^2,
\end{align*}
which is the lower bound of Theorem~\ref{thm:mainintro}.

The main challenge is the analysis of the elastic energy because it is not obvious why it should
scale like $O(\# \partial X)$ and not like $O(\# X)$. It will be shown in Section~\ref{sec:pairenergy} that (\ref{elastbd}) holds because of a cancelation argument which relies on the presence of certain reflection symmetries in the fcc lattice.

The bounds (\ref{defectbd}) and (\ref{structbd}) are considerably less involved, a
proof of the estimates is given in Sections~\ref{sec:struct} and \ref{sec:defect}.
\subsection{Bulk and surface energy} \label{sec:struct}
We show that $E_\mathrm{struct}$ can be estimated from below by a negative
bulk contribution and a positive surface energy.
\ft{Since $\Psi(r_1,r_2,r_3)=0$ if $\max_{i} r_i \geq 1+\alpha$ Proposition~\ref{prop:maxnhd}.1 implies that
\begin{align} \label{e3est}
e_3(x)\geq \begin{cases} 48\,\Psi(1,1,1) & \text{if } x \in \Xreg,\\
46\,\Psi(1,1,1) & \ft{\text{else.}}
\end{cases}
\end{align}}
By the definition of $E_\mathrm{struct}$
and assumption~(\ref{tbc}), we obtain the inequality
\begin{align*}
E_\mathrm{struct}(y)
\geq & \ft{48\,n\,\Psi(1,1,1)-2\,(n-\#\Xreg)\,\Psi(1,1,1)}
+\# \nnb\,V(1) + \# P(\sqrt{2})\, V(\sqrt{2})\\
&+\# P(\sqrt{8/3})\,V(\sqrt{8/3}) + \# P(\sqrt{3})\,V(\sqrt{3})
+\sum_{\lambda \in \Lambda_\mathrm{long}}
\#P(\lambda)\,V(\lambda),
\end{align*}

with $\nnb$ and $P$ defined in (\ref{eq:sdefn})
and Definition~\ref{defn:kbonds}.
The bounds in Lemma~\ref{lemma:cardinality} together with the relation $\#X_\mathrm{tco} +
\# X_\mathrm{co} = \# X_\mathrm{reg}$ \amend{32}{and the assumptions on $V$ and $\Psi$} imply that
\begin{align*}
E_\mathrm{struct}(y)
\geq & (48\,n-2(n-\#\Xreg))\,\Psi(1,1,1)+(m(1)\,n - (n-\# X_{12}))\,V(1)\\
&+ n\,m(\sqrt{2}) \,V(\sqrt{2})
 +\lf{2\# \Xtcub^2\, V(\sqrt{8/3}) + \left(m(\sqrt{3})\,\#\Xcub^2  + 18\,\#\Xtcub^2\right)\,V(\sqrt{3})}\\
& \lf{-C(n - \# X_\mathrm{reg})\alpha^{\frac{1}{4}}+ \sum_{\lambda >\sqrt{3}} m(\lambda)\,\left(nV(\lambda)
-C\alpha\,\# \partial X\,\lambda^{-10}\right)}
\\
\geq
&\estar\, n -\lf{2(n-\#\Xreg)\,\Psi(1,1,1)}+V(1)\,(\# X_{12}-n) \nonumber
+2\underbrace{\left(V(\sqrt{8/3})-3\,V(\sqrt{3})\right)}_{\geq \alpha^\frac{1}{2}} \#X_\mathrm{tco}^2\\
&+\lf{m(\sqrt 3)(n-\#\Xreg^2)V(\sqrt 3)}-C\,\alpha^\frac{1}{4}\,(n - \# X_\mathrm{reg})
-C\alpha\#\partial X
&\end{align*}

The assumptions on $V$ in Definition~\ref{defn:adpot} togehter with \eqref{eq:X2} and \eqref{s3bound} imply that

\begin{align*}
E_\mathrm{struct}(y) \geq & \estar\,n +\left(-V(1)-C\,\alpha^\frac{1}{4}-\lf{C\alpha^{\half}}\right)(n-\#X_{12})\\
&+\left(-\Psi(1,1,1)-C\,\alpha^\frac{1}{4}-\lf{C\alpha^{\half}}\right)\,(X_{12}-\# X_\mathrm{reg})
+c\alpha^\frac{1}{2}\lf{\# X_\mathrm{tco}} -C\alpha \#\partial X,
\end{align*}
which implies eqn.~(\ref{structbd}) if $\alpha$ is sufficiently small.
\subsection{Defect energy}\label{sec:defect}
The defect energy can be decomposed into several parts: interaction energies of pairs close to particles in $X\setminus X_\mathrm{reg}$, and contributions of pairs near $\partial X$:
\begin{align*}
E_\mathrm{defect}(y)= \sum_{k=1}^\infty\sum_{i=1}^3\sum_{p \in P_{0,k,i}} V(|y(p_2)-y(p_1)|)
\end{align*}
with
\begin{align*}
P_{0,k,1} &= \left\{p \in P_{0,k} \;: \; \exists x_p \in X \setminus X_\mathrm{reg} \text{ such that } \left|y(x_p)-y(p_1)\right| \leq 10(k+1)+3\right\},\\
P_{0,k,2} &= \left\{p \in P_{0,k}\setminus P_{0,k,1} \;: \; \exists x_p \in X_\mathrm{reg} \setminus X_\mathrm{co} \text{ such that } \left|y(x_p)-y(p_1)\right| \leq 10(k+1)+3\right\},\\
P_{0,k,3} &= P_{0,k} \setminus (P_{0,k,1} \cup P_{0,k,2}),\\
P_{0,k} &=\{p \in P_0 \; : \; k \leq |y(p_1) - y(p_2)| < k+1\}. \nonumber
\end{align*}
Clearly the sets $P_{0,k,i}$ form a partition of $P_{0,k}$.
The minimum distance bound~(\ref{eq:mindist}) implies that
\begin{align}
\label{p0k1} \# P_{0,k,1} &\leq C\,k^5 \,(n-\# X_\mathrm{reg}),\\
\label{p0k2}
\# P_{0,k,2} &\leq C\,k^5 \# (X_\mathrm{reg}\setminus X_\mathrm{co})\leq C\,k^5 \# \partial X,\\
\nonumber
\# P_{0,k,3} &\leq \amend{33}{\#}\left\{p \in P_{0,k} \;: \; \exists x_p \in X_\mathrm{co} \setminus X^2_\mathrm{reg} \text{ such that } \left|y(x_p)-y(p_1)\right| \leq 10(k+1)+3\right\}\\
&\leq  C\,k^5 \# (X_\mathrm{co}\setminus X^2_\mathrm{reg})
\leq C\,k^5 \# (X_\mathrm{reg}\setminus X^2_\mathrm{reg})\leq C k^5\,(n - \# X_\mathrm{reg}).\label{p0k3}
\end{align}
The last bound is due to (\ref{eq:X2}).

Inequalities (\ref{p0k1}) and (\ref{p0k3}) together with the estimates~(\ref{assump:vzero}),
(\ref{eq:assumptwo}) deliver the bound
\begin{align} \label{sp0k1}
\sum_{k=1}^\infty\sum_{p \in P_{0,k,i}} V(|y(p_2)-y(p_1)|) \geq
-C \alpha^\frac{1}{4} (\# X_\mathrm{reg}-n)
\end{align}
for $i \in \{1,3\}$.
It can be checked by inspection that $p\in P_{0,k,2}$ and (\ref{lamblist}) implies $|y(p_2)-y(p_1)|\geq \sqrt{\frac{7}{2}}$ and thus another application of the decay estimate~(\ref{eq:assumptwo}) together with (\ref{p0k2}) delivers the bound
\begin{align}
\label{sp0k2}
\sum_{k=1}^{\infty}\sum_{p \in P_{0,k,2}} V(|y(p_2)-y(p_1)|) \geq - C \alpha \#\partial X.
\end{align}
Estimate~(\ref{defectbd}) is a result of the combination of~(\ref{sp0k1}) and (\ref{sp0k2}).
\subsection{Elastic energy}\label{sec:pairenergy}
The main challenge is the demonstration that the elastic bulk contribution is non-negative. The
core of the argument is based on a cancelation effect induced by the reflection map $\kappa$
(Section~\ref{sec:pairs}). The details of this step can be found at the end of
Section~\ref{sec:linear}.

Recall the definition of $P_*$ in Definition~\ref{defn:kbonds}.
We split the elastic energy into short and long-range
contributions by defining
\begin{align*}
E_\mathrm{elast}(y) &= E_\mathrm{short}(y)+ E_\mathrm{long}(y) + E_\mathrm{med}(y)\\
E_\mathrm{short}(y) &= \sum_{q \in \nnb} (V(|y(q_2) - y(q_1)|)-V(1)),\\
E_\mathrm{long}(y) &= \sum_{\lambda \in \Lambda \setminus \{1\}} \sum_{p \in P_*(\lambda)}
(V(|y(p_2) - y(p_1)|)-V(\lambda)),\\
E_\mathrm{med}(y) &= \sum_{\lambda \in \Lambda} \sum_{p \in P(\lambda) \setminus P_*(\lambda)} (V(|y(p_2) - y(p_1)|)-V(\lambda)).
\end{align*}
The medium-range contributions are associated with pairs contained in the
neighborhoods of regular particles, which do not form the ends of paths.
If $\lambda \in \Lambda_\mathrm{med}$ and $p \in P(\lambda)$, then
Proposition~\ref{prop:discretelocal} and (\ref{eq:distortionineq}) imply that
$|V(|y(p_2) - y(p_1)|)-V(\lambda)| \leq C \alpha,$ and thus the inequality
\begin{align*}
E_\mathrm{med}(y) &\geq -C\alpha \sum_{\lambda \in
\Lambda_\mathrm{med}} \#(P(\lambda) \setminus P_*(\lambda)).
\end{align*}
The definition of $P_*(\lambda)$ and $P(\lambda)$ together with the minimum distance bound (\ref{eq:mindist}) implies that
$\# (P(\lambda)\setminus P_*(\lambda)) \leq C\# \partial X$ and one finds that
\begin{align*}
E_\mathrm{med}(y) \geq -C\alpha \#\partial X.
\end{align*}


To simplify the analysis of $E_\mathrm{elast}$ we apply eqn.~(\ref{pathnumb}) and write
the long-range energy contributions as a sum over paths:
\begin{align} \label{elastpath}
E_\mathrm{elast}(y)\geq E_\mathrm{short}(y)+\sum_{\lambda\in \Lambda\setminus\{1\}}
\sum_{\gamma \in  \hat \Gamma(\lambda)}
\hat M(\gamma)(V(|y(\gamma(\nu)) - y(\gamma(0))|)-V(\lambda))-C\alpha\#\partial X.
\end{align}
Let $(\Omega,\Phi,u)$ be the reference configuration associated with $\gamma$
\amend{34}{and let} $R \in SO(3)$ be the rotation which achieves the minimum in
(\ref{fjmineq}) for $s=2$. Note that $R$ and $\Omega$ depend on $\gamma$;
as it will turn out that this this dependency is irrelevant for the bounds,
we will mostly suppress it in our notation.
For the remainder of this section the shorthand $\eta_1=\Phi^{-1}(\gamma(0))$
and $\eta_2=\Phi^{-1}(\gamma(\nu))$ will be used. Expanding $V(|u(\eta_2) - u(\eta_1)|)$ one obtains
\begin{align}
& V(|u(\eta_2)-u(\eta_1)|) - V(\lambda) \nonumber\\
= &V'(\lambda)(|u(\eta_2)-u(\eta_1)|-\lambda) +
\frac{1}{2} V''(r(\gamma))(|u(\eta_2)-u(\eta_1)|-\lambda)^2,\label{Vexp}
\end{align}
where $r(\gamma)>0$ satisfies
\begin{eqnarray}\label{eq:rp}
\left|r(\gamma)-\lambda\right|\leq\left|\left|u(\eta_2)
-u(\eta_1)\right|-\lambda\right|\leq C\,\alpha\,\lambda
\end{eqnarray}
for a universal constant $C>0,$ provided $\alpha>0$ is sufficiently small.
The final inequality is due to the rigidity bound~(\ref{eq:distortionineq}).\amend{35}{}

Next, we define the distortion
\begin{eqnarray}\label{eq:deltadefn}
 \delta(\gamma)=R^T\,(u(\eta_2)-u(\eta_1))-\hat k(\gamma)\in\re^3.
\end{eqnarray}

By Proposition \ref{prop:distsum}, the distortion can be bounded by a sum of
edge length distortions in the following way:
\begin{align} \nonumber
 \left|\delta(\gamma)\right|^2\amend{36}{\leq}&
\left(\sum_{i=1}^\nu\left|\nabla u(\xi_i) (\mu(i)-\mu(i-1)))
-R\,\hat k(\gamma)\right|\right)^2\\
\leq&
C\amend{36}{\lambda^2}\sum_{q \in \nnb(\Omega)}\left|\left|y(q_2)-y(q_1)\right|-1\right|^2,
\label{eq:deltaest}
\end{align}
where the points $\xi_i\in \mathrm{interior}(\conv(\sigma_i))$ are arbitrary
and each simplex $\sigma_i \in \mathcal D$ has the property $\mu(i-1),\mu(i) \in \sigma_i$. The right-hand side of
(\ref{eq:deltaest}) is clearly an overestimation of the distortion, in which
the sum of edge length distortions contains an order of $\lambda^3$ terms.
An estimate of this form is sufficient because of the strong decay (\ref{assump:vfive}).

We now expand each term in the right-hand side of (\ref{Vexp}):
\begin{align*} 
 \left|u(\eta_2)-u(\eta_1)\right|=\lambda+
\frac{\delta(\gamma)\cdot \hat k(\gamma)}{\lambda}+e(\gamma)
\end{align*}
where $e(\gamma)\in\re$ is the second-order remainder term, satisfying
\begin{align} \label{remest}
\left|e(\gamma)\right|\leq C\left|\delta(\gamma)\right|^2,
\end{align}
for a universal constant $C>0.$ \amend{37}A second order
expansion of the energy contribution of $\gamma$ then takes the form
\begin{align*}
  V(\left|u(\eta_2)-u(\eta_1)\right|)=V(\lambda)
+\left(\frac{\delta(\gamma) \cdot \hat k(\gamma) }{\lambda}
+e(\gamma)\right)V'(\lambda)+\half\left(\frac{\delta(\gamma)\cdot\hat k(\gamma)}{\lambda}
+e(\gamma)\right)^2V''(r(\gamma)).
\end{align*}

If $\lambda\in\left\{\sqrt{2},\sqrt{3}\right\},$ then (\ref{eq:rp}) and
assumption (\ref{assump:vfourprimetwo}) on the pair potential $V$ implies that
$\left|V''(r(p))\right|\leq\alpha^\frac{1}{4}.$

With this notation eqn.~(\ref{elastpath}) takes the form
\begin{align}\label{eq:pairTaylor}
E_\mathrm{elast}(y)\geq \mathcal{R}+\g - C\alpha \#\partial X,
\end{align}
where
\begin{align} \label{eq:r}
\mathcal{R} &=\sum_{\lambda\in \Lambda \setminus\{1\}} \sum_{\gamma \in \hat\Gamma(\lambda)}\hat M(\gamma) \left[ e(\gamma)\, V'(\lambda)
+\half\left(\frac{\delta(\gamma)\cdot \hat k(\gamma)}{\lambda}+e(\gamma)\right)^2
V''(r(\gamma))\right],\\
\label{eq:g}
 \g &=E_\mathrm{short}(y)+\sum_{\lambda \in \Lambda\setminus\{1\}} \sum_{\gamma \in \hat \Gamma(\lambda)} \hat M(\gamma)
\frac{\delta(\gamma)\cdot\hat k(\gamma)}{\lambda}V'(\lambda).
\end{align}
We will demonstrate that
\begin{align} \label{eq:epsilon}
\mathcal{R} &\geq -C\alpha^\frac{1}{4}\sum_{q\in \nnb}\left|\left|y(q_2)-y(q_1)\right|-1\right|^2,\\
\g & \geq c\sum_{q\in \nnb}\left|\left|y(q_2)-y(q_1)\right|-1\right|^2- C\alpha\#\partial X. \label{gbound}
\end{align}
Combining (\ref{eq:pairTaylor}), (\ref{eq:epsilon}) and (\ref{gbound}) gives (\ref{elastbd}).

\subsubsection{Proof of inequality~(\ref{eq:epsilon})}\label{sec:second}
We use the rigidity estimate (\ref{eq:distortionineq}) and inequality (\ref{eq:deltaest})
to reduce $\mathcal{R}$ to a localized quadratic
sum of edge distortions.

Taking the modulus of each term in (\ref{eq:r}) and using (\ref{remest}), \amend{38}{\eqref{assump:vfive} and Lemma~\ref{lemma:assumptions}} we find
\begin{align*}
\left|\r\right| \leq
&C\alpha^\frac{1}{4}\sum_{\lambda\in\Lambda\setminus\{1\}}\sum_{\gamma\in\hat\Gamma(\lambda)}\hat M(\gamma)
\left(\left|\delta(\gamma)\right|^2\lambda^{-9}+\sum_{i=2}^4
\left|\delta(\gamma)\right|^i\lambda^{-10}\right).
\end{align*}
Recall the definition of the set $\nnb(\Omega)$ in (\ref{eq:sonedomain}).
Estimate (\ref{eq:distortionineq}) implies that $|\delta(\gamma)|\leq
C\lambda$ and together with (\ref{eq:deltaest}) one finds that
\begin{align*}
\left|\r\right| \leq&C\alpha^\frac{1}{4}\sum_{q\in \nnb}
\left|\left|y(q_2)-y(q_1)\right|-1\right|^2
\sum_{\lambda\in\Lambda}\amend{39}{\lambda^{-6}}\sum_{\gamma \in \hat \Gamma(\lambda)}\hat M(\gamma)
\chi_{\nnb(\O)}(q) \\
\leq  & C\alpha^\frac{1}{4}\sum_{q\in \nnb}
\left|\left|y(q_2)-y(q_1)\right|-1\right|^2
\sum_{\lambda\in\Lambda}\amend{39}{\lambda^{-6}} \sum_{p \in P(\lambda)}
\underbrace{\sum_{\atop{\gamma \in \hat \Gamma(\lambda)}
{p = (\gamma(0),\gamma(\nu))}} \hat M(\gamma)}_{=1}\\
& \qquad\qquad \times \max\{\chi_{\nnb(\O(\gamma))}(q) \; | \;
p = (\gamma(0),\gamma(\nu))\} \\
\leq&C\alpha^\frac{1}{4}\sum_{q\in \nnb}\left|\left|y(q_2)-y(q_1)\right|-1\right|^2,
\end{align*}
which is (\ref{eq:epsilon}).
The final estimate is due to the injectivity of each
discrete imbedding, which implies the combinatorial estimate
\begin{eqnarray}\label{eq:ballinverse}
  \#\left\{p\in P(\lambda)
    \;:\; q\in \nnb(\O(\gamma)) \text{ for some } \gamma\in \Gamma(\lambda)
    \text{ such that } p=(\gamma(0),\gamma(\nu)) \right\}\leq C m(\lambda) \lambda^3
\end{eqnarray}
for each $\lambda\in\Lambda$ and $q\in \nnb.$
\subsubsection{Proof of inequality~(\ref{gbound})}\label{sec:linear}
The aim of this section is to prove
\begin{align}
  \g =&E_\mathrm{short}+\sum_{\lambda \in \Lambda\setminus \{1\}} \sum_{\gamma \in \hat \Gamma(\lambda)} \hat M(\gamma)
\frac{\delta(\gamma)\cdot\hat k(\gamma)}{\lambda}V'(\lambda)
\geq c\sum_{q\in \nnb} \left|\left|y(q_2)-y(q_1)\right|-1\right|^2
-C\alpha\#\partial X.
\label{eq:linearbound}
\end{align}
Inequality (\ref{eq:linearbound}) is not a direct consequence of simple
estimates, since the left-hand is a linear function of pair distortions,
whilst the right-hand side contains quadratic terms.
As a first step, the sum $\g$ is written as a localized sum of nearest neighbor quantities.
For the short-range interaction it is easy to see that
$E_\mathrm{short}(y)= I_1+I_2$ with
\begin{align}
I_1 = & V'(1)\,\sum_{q \in \nnb} (|y(q_2)-y(q_1)| - 1),\\
I_2 = & \sum_{q \in \nnb} \frac{1}{2}V''(\rho_q)(|y(q_2)-y(q_1)|-1)^2
\end{align}
where $\rho_q \in [1-\alpha,1+\alpha]$.
Assumption~(\ref{assump:vtwo}) implies that
\begin{align} I_2 \geq \sum_{q \in \nnb} \frac{1}{2}(|y(q_2)-y(q_1)|-1)^2.
\label{coercivity}
\end{align}
To localize the long-range interactions each vector
$y(p_2(\gamma))-y(p_1(\gamma))$ is decomposed into a sum over edges, c.f. (\ref{eq:pathsum}).  To this end, for each $\gamma\in\glab$ and $q\in
\nnb,$ define the indicator function $g_{q}(\gamma)\in\left\{0,1\right\}$ by
\begin{eqnarray*}
 g_{q}(\gamma)&=&\left\{\begin{array}{ll}
                         1 & \textnormal{if} \ q\in\gamma,\\
0 & \textnormal{otherwise.}
                        \end{array}\right.
\end{eqnarray*}
The injectivity of $\gamma$ implies that
\begin{eqnarray}\label{eq:pathsum}
 y(p_2(\gamma))-y(p_1(\gamma))=\sum_{q\in\gamma}(y(q_2)-y(q_1))
=\sum_{q\in \nnb}g_{q}(\gamma)\,(y(q_2)-y(q_1)).
\end{eqnarray}
If $\gamma\in\glab$ and $q\in \nnb$ is such that $q\in\gamma,$ then we may choose
a fixed simplex $\sigma(q)\in\d$ with the property $q_1,q_2 \in \Phi(\sigma(q))$
and define
$F_{q}=\left.\nabla u\right|_{\sigma(q)}\in M^{3\times 3}$. Let furthermore $R_{q}\in
SO(3)$ be a minimizer of $R\mapsto \left|F_q-R\right|$
subject to the constraint $(y(q_2)-y(q_1))=r_q \,R\,\hat k(q)$,
with $r_q=\left|y(q_2)-y(q_1)\right|\in\left[1-\alpha,1+\alpha\right].$
Then we obtain
\begin{align*}
 y(p_2(\gamma))-y(p_1(\gamma))= \sum_{q \in \nnb} g_q(\gamma) r_q R_q \hat k(q).
\end{align*}

Notice that for every $R \in SO(3)$ the equation
$$\amend{40}{R^T} = \Id + \frac{1}{2}(R^T-R)-\frac{1}{2}(\Id-R^T)(\Id-R)$$
holds. This implies that
\begin{eqnarray}\label{eq:rqt}
R^T=(R^T R_{q})R^T_{q}=(\I+A_{q}+G_{q})R^T_{q},
\end{eqnarray}
where $A_{q}$ is a skew symmetric matrix and $G_{q}\in M^{3\times 3}$ satisfies
$\left|G_{q}\right|\leq C\left|R-R_{q}\right|^2.$ Thus,
\begin{align}\label{eq:jsplit}
\g&= I_2+J_1 + J_2+ J_3-J_4
\end{align}
with
\begin{align*}
  J_1 = &\sum_{q\in \nnb}\left(V'(1) r_q +\sum_{\lambda\in\Lambda\setminus\{1\}}\sum_{\gamma\in\hat \Gamma(\lambda)}
  \hat M(\gamma)\,r_q\,\lf{g_{q}(\gamma)}\, \hat k(\gamma)\cdot \hat k(q) W(\lambda)\right),\\
  J_2 = & \sum_{q\in \nnb}\sum_{\lambda\in\Lambda\setminus\{1\}}\sum_{\gamma\in\hat \Gamma(\lambda)}\hat M(\gamma)\,r_q\, \lf{g_{q}(\gamma)}\, \hat k(\gamma)\cdot A_{q}\,\hat k(q) W(\lambda),\\
  J_3 = & \sum_{q\in \nnb}\sum_{\lambda\in\Lambda\setminus\{1\}} \sum_{\gamma\in\hat \Gamma(\lambda)}\hat M(\gamma)\, r_q\, \lf{g_{q}(\gamma)}\, \hat k(\gamma)\cdot G_{q}\,\hat k(q)\,W(\lambda),\\
  J_4 = &V'(1) \# \nnb + \sum_{\lambda\in\Lambda \setminus\{1\}}\sum_{p \in P(\lambda)}\lambda^2W(\lambda),
\end{align*}
where $W(\lambda) = \frac{1}{\lambda}V'(\lambda)$. Estimate~(\ref{gbound}) is a consequence
of (\ref{coercivity}) and the inequalities
\begin{align}
\label{J1bound} J_1 \geq & -C\,\alpha\,\#\partial X,\\
\label{J2bound} J_2 \geq & -C\, \alpha\, \#\partial X,\\
\label{J3bound} J_3 \geq& -C\alpha\sum_{q\in \nnb}\left|\left|y(q_2)-y(q_1)\right|-1\right|^2,\\
\label{J4bound} J_4 \geq& -C\alpha^\frac{1}{4}(n-\# X_\mathrm{reg}) - C \alpha \# \partial X,
\end{align}
which will be established below.
\subsubsection*{Analysis of $J_1$ and $J_4$}
The bounds of the sums $J_1$ and $J_4$ are a consequence of the equilibrium
condition (\ref{eq:renormmin}).
\begin{align*}
J_4=&V'(1) \# \nnb+\sum_{\lambda\in\Lambda\setminus\{1\}}\sum_{p\in
  P(\lambda)}\lambda^2W(\lambda) \\
=&\sum_{\lambda\in\Lambda} \left(\# P(\lambda)
-\frac{m(\lambda)}{m(1)}\# \nnb\right)\lambda^2W(\lambda)
+\frac{1}{m(1)}\sum_{q\in \nnb}
\underbrace{\sum_{\lambda\in\Lambda}m(\lambda)
\lambda^2W(\lambda)}_{=0 \text{ by (\ref{eq:renormmin})}}\nonumber\\
= &\sum_{\lambda\in\Lambda}
\left(\# P(\lambda)-\frac{m(\lambda)}{m(1)}\# \nnb\right)
\,\lambda^2W(\lambda)\\
\geq&-C\,|V'(\sqrt{2})|\,(n-\# X_\mathrm{reg}) - \sqrt{3} V'(\sqrt{3})\,\left( C(n-\# X_\mathrm{reg})
-6\# X_\mathrm{tco}\right) - C\alpha\#\partial X\\
\geq &-C|V'(\sqrt{2})|\,(n-\# X_\mathrm{reg}) - C \alpha \# \partial X\amend{42}{+}6\sqrt{3}V'(\sqrt{3}) \#
X_\mathrm{tco}.
\end{align*}
The penultimate equation is a consequence of assumptions~(\ref{assump:Vlip}), (\ref{assump:vfive}), and the bounds in Lemma~\ref{lemma:cardinality}. Together with (\ref{Vpsign}) this shows that (\ref{J4bound}) holds.

Next we establish a bound on
\begin{eqnarray*}
J_1&:=&\sum_{q\in \nnb}r_q\left(V'(1)+\sum_{\lambda\in\Lambda\setminus \{1\}}\sum_{\gamma\in\glab(\lambda)}
  \hat M(\gamma) \amend{43}{g_q(\gamma)} W(\lambda)\,\hat k(\gamma)\cdot \hat k(q)\right).
\end{eqnarray*}
Notice that the uniqueness of discrete imbeddings up to rotation and
translation implies that $J_1$ is independent of the choice of the reference
configuration.

First we demonstrate that for $q \in \nnb$ and $\lambda \in \Lambda$ the inequality
\begin{align}
\label{lambdabound}
\sum_{\gamma \in \hat \glab(\lambda)} \hat M(\gamma) \, g_q(\gamma)\, \hat k(\gamma)\cdot \hat k(q)\leq m(\lambda) \lambda^2
\end{align}
holds.

Recall eqn.~(\ref{eq:basisdefn}) and the convention $v \in B$ if there exists $i\in \{1,2,3\}$ such that
$v = Be_i$.
To eliminate the localization function $g_q$ we bound the number of paths $\gamma$ with
the properties $\hat k (\gamma) = k$ and $q =(\gamma(i),\gamma(i+1))$ for some $i$ from above by
\begin{align} \label{covct}
\sum_{B \in \mathcal B} f(k,B)a_k(v,B),
\end{align}
where $v= \hat k(q)$, the integer coefficients
$a_k(v,B) \in \{0,1,2,\ldots\}$ satisfy
\begin{align*} 
 k=\sum_{w\in B}a_k(w,B)w,
\end{align*}
and $f(k,B)\geq 0$ is given by
$$f(k,B) = \#\{B' \in \mathcal B \; : \; \mu(k,B')= \mu(k,B)\}^{-1}.$$
Expression~(\ref{covct}) is only an upper bound since not every reference path $\mu$
necessarily corresponds to  a label path $\gamma$.
It is easy to see that
\begin{eqnarray}\label{eq:kbasis}
 a_k(v,B)=\left\{ \begin{array}{rl} k^TB^{-T}B^{-1}v & \text{ if } v \in B,\\
0 & \text{ else,} \end{array} \right.
\end{eqnarray}
and $f(k,B) \hat M(\gamma) = \frac{1}{120}$ if $\mu_\gamma \in \Gamma[B]$, cf.~(\ref{eq:mdefn}).

These considerations lead to the inequality
\begin{align} \label{latineq}
\sum_{\gamma \in \hat \glab(\lambda)}\hat M(\gamma) \, g_q \,\hat k(\gamma)\cdot \hat k(q)
\leq  \frac{1}{120} \sum_{|k|=\lambda}  \sum_{B \in \mathcal B} a_{k}(v,B)k\cdot v
\end{align}
for all $v \in \sx \cap \lf{\cfcc}$.
Estimate~(\ref{lambdabound}) follows now from
\begin{lemma}\label{lemma:lambda}
For all $\lambda\in\Lambda,$ $B\in\b,$ and  $v\in B$ the identity
\begin{eqnarray}\label{eq:lambda}
 \sum_{\substack{k\in\lf{\cfcc}\\ \left|k\right|=\lambda}}a_k(v,B)k\cdot v=\frac{1}{3}m(\lambda)\lambda^2
\end{eqnarray}
holds, where $a_k(v,B)$ is the coefficient defined by (\ref{eq:kbasis}) and
$m(\lambda)$ is the number of lattice vectors of length $\lambda$, c.f. (\ref{eq:mldefn}).
\end{lemma}
\begin{proof}
See appendix.
\end{proof}
Define next
$$\mathcal D(\lambda) =\left\{ q \in \nnb \; : \; \sum_{\gamma \in \hat \glab(\lambda)}
\hat M(\gamma) \, g_q\,\hat k(\gamma)\cdot \hat k(q)< m(\lambda) \lambda^2\right\}.$$

Thanks to (\ref{lambdabound}) and the trivial bound $r_q \leq 2$ we find that
\begin{align*}
|J_1| \leq \biggl|\sum_{q\in \nnb}r_q\underbrace{\sum_{\lambda\in\Lambda}W(\lambda) \lambda^2 m(\lambda)}_{=0 \text{ by } (\ref{eq:renormmin})} \biggr|
+ 2\amend{44}{\sum_{\lambda\in\Lambda\setminus\{1\}}}\left|W(\lambda)\right|\, \lambda^2 m(\lambda) \# \mathcal D(\lambda)
\end{align*}
The minimum distance bound~(\ref{eq:mindist}) implies that $\# \mathcal D(\lambda)\leq C \lambda^3 \# \partial X$ for some absolute constant $C$. Together $W(\lambda) \leq C \alpha \lambda^{-10}$ one obtains the estimate
$$ |J_1| \leq C\,\alpha\,\#\partial X\,\sum_{k \in
     \cfcc\setminus \{0\}} |k|^{-5}\leq C\,\alpha\,\#\partial X,$$
which is (\ref{J1bound}).
\subsubsection*{Analysis of $J_3$ - application of rigidity estimates}
We obtain a lower bound on
\begin{eqnarray*}
 J_3:=\sum_{q\in \nnb}\sum_{\lambda\in\Lambda}\sum_{\gamma\in\hat\Gamma(\lambda)}
 \hat k(\gamma) \cdot \hat M(\gamma)\,r_q\,\lf{g_{q}(\gamma)}\,G_{q}\, \hat k(q)\,W(\lambda)
\end{eqnarray*}
by applying the rigidity estimates of Propositions \ref{prop:distsum} and \ref{prop:distsum2}.

For each $q\in \nnb(\Omega)$ (c.f. (\ref{eq:sonedomain})), let $\sigma(q)\in\d$ be \lf{the simplex which was chosen in the construction of $R_q.$  Then, }$\eta_1,\eta_2 \in \sigma(q)\subset\O$.  Using the bound
$\left|G_{q}\right|\leq C\left|R-\lf{R_q}\right|^2,$ and noting that for each simplex
$\sigma$ the number of bonds $q$ such that $\sigma(q) = \sigma$ is bounded by 6, we obtain
\begin{eqnarray*}
 \sum_{q\in \nnb}\left|\lf{g_{q}(\gamma)}\,G_{q}\right|&\leq&C\sum_{q\in \nnb}
g_{q}\left|R-R_{q}\right|^2\nonumber\\
&\leq&C\sum_{q\in \nnb}\underbrace{\lf{g_{q}(\gamma)}}_{\leq 1}\left[\left|\left.\nabla
u\right|_{\sigma(q)}-R\right|^2+\left|\left.\nabla u\right|_{\sigma(q)}
-R_{q}\right|^2\right]\\
&\leq&C\left(\|\nabla u-R\|_{L^2(\O)}^2
+\|\nabla u-\lf{R_q}\|_{L^2(\O)}^2\right)\\
&\leq&C\sum_{q\in \nnb(\O)}
\left|\left|y(q_2)-y(q_1)\right|-1\right|^2.
\end{eqnarray*}
The final inequality is due to (\ref{eq:ltwo}), (\ref{eq:ltwoq}) and Lemma~\ref{lemma:wcell}.
This implies
\begin{eqnarray}\label{eq:Jthree}
\left|J_3\right|&\leq&C\sum_{\lambda\in\Lambda}\lambda
\left|W(\lambda)\right|\sum_{p\in P(\lambda)}
\sum_{q\in \nnb(\Omega_p)}\left|\left|y(q_2)-y(q_1)\right|-1\right|^2
\underbrace{\sum_{\atop{\gamma\in\hat\Gamma(\lambda)}{\gamma(0)=p_1, \gamma(\nu) = p_2}} \hat M(\gamma)}_{=1 \ \textnormal{by} \ (\ref{pathnumb})}\nonumber\\
&\stackrel{(\ref{eq:ballinverse})}{\leq}
&C\sum_{\lambda\in\Lambda}m(\lambda)\lambda^4\left|W(\lambda)\right|\sum_{q\in \nnb}
\left|\left|y(q_2)-y(q_1)\right|-1\right|^2\nonumber\\
&\leq&C\alpha\sum_{q\in \nnb}\left|\left|y(q_2)-y(q_1)\right|-1\right|^2
\end{eqnarray}
since (\ref{assump:vfourprimetwo}) and (\ref{assump:vfive}) imply that $\left|W(\lambda)\right|\leq C\alpha\lambda^{-8}$ for all $\lambda\in\Lambda.$  Thus, equation (\ref{J3bound}) has been established.
\subsubsection*{Analysis of $J_2$ - pairwise cancelation of terms}
We will show that
\begin{align} \label{J2vanish}
J_2 =\lf{\sum_{\lambda\in\Lambda\setminus\{1\}}}\sum_{q\in \nnb}\sum_{\gamma\in\hat\Gamma}\hat k(\gamma)
\cdot \hat M(\gamma)\,r_q\,g_{q}(\gamma)\,\lf{A_{q}(\gamma)}\,\hat k(q) W(\lambda)\geq \lf{- C \, \alpha \,
\#\partial X.}
\end{align}
This is mainly a consequence of the observation that thanks to the skewness of
$A_{q}$ nonzero contributions \amend{45}{} are only generated by terms where $\lf{\hat k(\gamma)}$
and $\hat k(q)$ are not parallel. To treat mixed terms where $p \neq q$ we collect
for each $q$ those pairs $p$ such that the sum is parallel to $\hat k(q)$.

First we recall the label paths $\hat\Gamma_*$ (cf. Def.~\ref{def:labpath}) which are invariant under
the action of the reflections $\hat \kappa$ and define
\begin{align*}
J_2^*= \sum_{q\in \nnb}r_q\,\sum_{\gamma\in\hat\Gamma_*} \hat k(\gamma) \cdot
\hat M(\hat \kappa_q(\gamma))\, g_q(\lf{\hat \kappa_q(\gamma)})\,A_{q}(\lf{\hat \kappa_q(\gamma)})\,\hat k(q)\,W(\lambda(\hat \kappa_q(\gamma))).
\end{align*}

According to Remark~\ref{fixeddomain} we can assume that all paths $\gamma'$ in the orbit $\mathcal O\left(\gamma\right)$ share the same reference configuration.
Therefore
\begin{align}\label{Ainvl}
A_q(\gamma) = A_q(\hat \kappa_q(\gamma)), \;
\hat M(\gamma)  = \hat M(\hat \kappa_q(\gamma)), \;
g_q(\gamma) = g_q(\hat \kappa_q(\gamma)).
\end{align}
\ft{Thanks to (\ref{Ainvl}) the sum $J_2^*$ can be written as}
\begin{align*}
J_2^* = \sum_{q\in \nnb}r_q\,\sum_{\gamma\in\hat\Gamma_*} \hat k(\hat \kappa_q(\gamma)) \cdot
\hat M(\hat \kappa_q(\gamma))\, g_q(\hat \kappa_q(\gamma))\,A_{q}\lf{(\hat \kappa_q(\gamma))}\,\hat k(q)\,W(\lambda(\hat \kappa_q(\gamma))).
\end{align*}
Adding the two expressions for $J^*_2$ one arrives at the representation
\begin{align*}
J_2 =(J_2-J_2^*)  + \sum_{q\in \nnb}r_q\,\sum_{\gamma\in\hat\Gamma}\frac{1}{2}
\hat M(\hat\kappa_q(\gamma))\,g_q(\hat\kappa_q(\gamma))\,\left[\hat k(\gamma)+ \hat k(\hat \kappa_q(\gamma))\right]\cdot A_{q}(\hat\kappa_q(\gamma))\,\hat k(q)\, \,W(\lambda(\hat\kappa_q(\gamma))).
\end{align*}
Equation~(\ref{parallel}) implies that
$\left[\hat k(\gamma)+\lf{\hat k(\hat \kappa_q(\gamma))}\right]\cdot A_{q}\,\hat k(q)=0$ since $A_q=\lf{A_q(\hat\kappa_q(\gamma))}$ is skew-symmetric, and thus
$$\lf{J_2 = J_2- J_2^*.}$$
The proof of the estimate $|J_2| \leq C\, \alpha \, \#\partial X$ is analogous to the
proof in Section~\ref{sec:defect}.
\end{proof}
\subsection{Proof of Corollary~\ref{cor:per}}
\begin{proof}
We will only consider the periodic case, compactly supported perturbations can be treated analogously.
Let $Z$ and $\lf{\c}$ be such that $\lf{\frac{1}{\#Y}} \,E_{\mathcal A}(Z) \leq \Efcc(1)$. A straight forward generalization of the proof of Proposition~\ref{prop:mindist} allows us to construct an \lf{$\c$-periodic} configuration $\tilde Z\subset Z$ such that
$$\frac{1}{\#\tilde Y} \,E_{\tilde\a}(\tilde Z)\leq \frac{1}{\#Y} \,E_{\a}(Z)$$
and $\tilde Z$ satisfies the minimum distance bound
$$ \min_{\atop{y,y' \in \tilde Z}{y\neq y'}}|y-y'| \geq 1-a.$$
Following the steps of the proof of the Theorem~\ref{thm:mainintro} one obtains the bound
\begin{align*}
E_{\tilde \a}(\tilde Z)\geq \estar\,\#\tilde Y+C\sum_{y \in \tilde Y} \sum_{\atop{y' \in \tilde Z}{||y-y'|-1|\leq a}}\left|\left|y-y'\right|-1\right|^2+ C\alpha^\frac{1}{2}\#\partial \tilde Y,
\end{align*}
where $\partial \tilde Y$ is defined analogously to Definition~\ref{defn:regular}. On the other hand, by construction $E_{\tilde \a}(\tilde Z)\leq \estar\,\lf{\#\tilde Y}$; thus $\partial \tilde Y =\emptyset$, and  $\left|y-y'\right|=1$ for all $y,y' \in \tilde Z$ such that $||y-y'|-1|\leq a$.

Proposition~\ref{prop:refconfig} implies that there exists a map $u: \R^3 \to \R^3$ such that $\tilde Z=u(\cfcc)$. Since $|u(\eta)-u(\eta')|=1$ for $|\eta-\eta'|=1$ we conclude that
$u(\eta) = Q\eta + t$ for some $Q \in SO(3)$, $t \in \R^3$ and all $\eta \in \R^3$ by Proposition~\ref{prop:distsum}.
\end{proof}
\section{Appendix}
\subsection*{Proof of Proposition~\ref{prop:discretelocal}}
Let $\sigma \in \mathcal U$ be a tetrahedron such that
$0,\xi \in \sigma $. Clearly $\sigma \subset Q \cup \{0\}$. We will show next that
$\Phi(\sigma) \subset N(x')$.
Indeed, if $\eta \in \sigma \setminus\{0,\xi\}$, then
eqn.~(\ref{eq:qdist}) implies that
$$||y\circ\Phi(\eta)-y\circ\Phi(\xi)| -1|\leq 2 \eps.$$
Theorem~\ref{thm:kissing} implies that $\Phi(\eta) \in N(x')$ for sufficiently
small $\eps$ since $x' \in X_\mathrm{reg}$.

Define next $\sigma'=(\Phi')^{-1}\circ \Phi(\sigma)\in \mathcal U$.
The rotation $T \in SO(3)$ is
characterized by the requirements
\begin{align}
\label{eq:cubinv} T \sigma'+\xi &= \sigma, \\
\label{eq:edgeinv}
T\left((\Phi')^{-1}(x)\right) +\xi &=0.
\end{align}

Define $\eta' = T^{-1}(\eta-\xi)$.
We will show later that
\begin{align} \label{eq:phieq}
\Phi(\eta) = \Phi'(\eta')
\end{align}
holds for all $\eta \in \sigma$. Eqn.~(\ref{eq:phieq}) implies that
\begin{align*}
|y(x)+R(T\eta' + \xi)-y(x') -R'\eta' |\leq 2\eps
\end{align*}
for all $\eta \in \sigma'$. Since $|y(x)-y(x')-R\xi| \leq \eps$ one obtains the bound
\begin{align*}
|(T-R^{-1} R')\eta'| \leq 3 \eps,
\end{align*}
for all $\eta' \in \sigma'$. Let now $M$ be a $3\times 3$ matrix whose columns are the vectors connecting the origin with the remaining 3 vertices of $\sigma'$. Without loss of generality we can assume that
$M=(b_1\, b_2\, b_3)$ where $b_i$, $i \in \{1,2,3\}$ are the basis vectors defined after (\ref{eq:fccdefn}). A simple explicit calculation shows that
$|M^{-1}| = \sqrt{2}$. This implies that
\begin{align*}
\left|T-R^{-1} R'\right| \leq 3 \sqrt{3}\, \eps
\end{align*}
and thus (\ref{eq:rotcomp}) holds.

It can be checked by inspection that the set
$\sx \cap (T(Q'\cup \{0\}) + \Phi^{-1}(x'))$ has at least 5 elements. Estimate~(\ref{eq:qdist})
and Theorem~\ref{thm:kissing} together imply that the set has precisely 5 elements.
This implies that $\#A=6$ holds.

Now we establish (\ref{eq:nnunique}). If $\eta \in A$, then
\begin{align*}
 &\left|y \circ \Phi(\eta) -y\circ \Phi'(\eta')\right|\\
\leq&
 \left|y\circ \Phi(\eta)-y(x)-R\eta\right|+|y(x)-y(x') + R\eta-R'\eta'|
+\left|y(x')+R'\eta'-y\circ\Phi'(\eta')\right|\\
\leq & 2\eps + \left|y(x)-y(x') + R\left(\eta-(R^{-1}R'-T+T)T^{-1}(\eta-\xi)\right)\right|\\
\leq & 2\eps + \left|y(x)-y(x') + R\xi\right| +\left|(R^{-1}R'-T)T^{-1}(\eta-\xi)\right|\\
\leq & 9 \eps
\end{align*}
by (\ref{eq:qdist}) and (\ref{eq:rotcomp}). Thus, if $9 \eps<1-\alpha$ then
$y \circ \Phi(\eta) =y\circ \Phi'(\eta')$. This is  (\ref{eq:nnunique}).

We finish the proof by establishing (\ref{eq:phieq}). Enumerate the vertices of the simplex so that
$\sigma=\{\sigma_1\ldots \sigma_4\}$ and assume that $\sigma_1=0$, $\sigma_2=\xi$.
Eqn.~(\ref{eq:edgeinv}) implies that (\ref{eq:phieq}) holds if $\eta\in \{\sigma_1, \sigma_2\}$. We will demonstrate that
$$ \Phi(\sigma_3) = \Phi'\left(T^{-1}(\sigma_4-\xi)\right)$$
implies
\begin{align} \label{eq:qdist1}
\min_{R' \in SO^-(3)}|R' \eta' + y(x') - y\circ\Phi'(\eta')|\leq 2\veps,
\end{align}
with $SO^-(3)=O(3) \setminus SO(3)$. For sufficiently small $\eps$ this contradicts (\ref{eq:qdist}) because the convex hull of $\sigma$ has positive volume.

To see that (\ref{eq:qdist1}) holds we assume without loss of generality
that $R=\Id$ and define the reflection
$$ R' = T\,(\Id - 2 (\sigma'_4-\sigma'_3) \otimes (\sigma'_4-\sigma'_3)).$$
Clearly $R' \in SO^-(3)$ and $R'\sigma_4' = T\sigma_3' = \sigma_3-\sigma_2$. Next, one calculates
\begin{align*}
& |R'\sigma_4' + y(x') - y\circ\Phi'(\sigma_4')|
= |\sigma_3-\sigma_2 + y(x) + y(x') -y(x)- y\circ \Phi(\sigma_3)|\\
\leq& |y(x) - y \circ \Phi(\sigma_3)+\sigma_3|+|y(x')-y(x) -\sigma_2| \leq 2\eps
\end{align*}
by (\ref{eq:qdist}) since $R=\Id$. Thus (\ref{eq:qdist1}) holds.
\subsection*{Proof of Proposition~\ref{prop:locref}}
Let $Q \in \{\cub, \tcub\}$, $R \in SO(3)$, $\Phi:Q \to X$ be the associated domains, rotations and maps from Proposition~\ref{prop:maxnhd}. Depending on $Q$ we select
$\c \in \{\cfcc, \chcp\}$ and the units $\mathcal U$ accordingly. Assume furthermore that
$\tau \in \mathcal U$ is an
octahedron such that $\tau \cap Q$ is a square.
For $\xi \in \tau \cap Q$ we define $x' = \Phi(\xi)$ and assume that
$Q' \in \{\cub, \tcub\}$, $\Phi':Q \to X$ and $R', T \in SO(3)$ are the associated domains, maps and rotations from Proposition~\ref{prop:maxnhd} and Proposition~\ref{prop:discretelocal}.

We extend $\Phi$ to \lf{$Q \cup \tau$} by defining
$$\Phi(\eta) = \Phi'(T^{-1}(\eta-\xi))$$
if $\eta \in \tau$ is the outmost vertex, i.e. $|\tau|=\sqrt{2}$.
Define furthermore \lf{$\Omega = \conv(Q) \cup \conv(\tau)$} and the interpolation
$u: \Omega \to \R^3$ according to Definition~\ref{defn:admap}.
The triple $(\Omega,\Phi,u)$ satisfies the requirements of Definition~\ref{defn:discretelocal} if
we show that the map $\Phi(\eta)$ does not depend on the choice of $\xi$.
Independence holds if we establish the bound
\begin{align}\label{eq:x3eq}
|y(x) +R\eta-y\circ \Phi(\eta)|\leq 8\eps
\end{align}
and choose $\alpha$ so small that $\eps < \frac{1}{16}$.

An application of the triangle inequality to the left hand side of
(\ref{eq:x3eq}) yields
\begin{align*}
 &\left|y(x)+R\eta-y\circ \Phi(\eta)\right|\\
\leq&
 \left|y(x)+R\eta'-y(x')\right|
+\left|R(\eta-\eta')+y(x')-y\circ\Phi'(T^{-1}(\eta-\eta'))\right|
+|(R-T^{-1}R')(\eta-\eta')|.
\end{align*}
Eqn.~(\ref{eq:qdist}) implies that the first two terms are bounded by
$\eps$, eqn.~(\ref{eq:rotcomp}) implies that the third term
is bounded by $6\eps$.

We repeat this procedure 5 more times
until we end up with the Lipschitz domain
$$\Omega= \bigcup_{\atop{\tau \in \mathcal U}{0 \in \tau}} \conv(\tau).$$

Equation (\ref{eq:snd}) is an immediate consequence of the construction.

\subsection*{Proof of Proposition \ref{prop:refconfig}}
We define $\Omega_l = l \cub$ and construct inductively reference configurations $(\Omega_l,\Phi_l,u_l)$ for $l \in \{1\ldots s\}$ such that $\Phi_l(\eta_\mathrm{center}) = x$
with $\eta_\mathrm{center}= [l/2]\sqrt{2}(1,0,0)^T\in \Omega_l \cap \cfcc$ and
\begin{align}
\label{eq:nablasoind}
 \|\dist(\nabla u, SO(3))\|_{L^\infty(\Omega_l)} \leq C \alpha
\end{align}
for some universal constant $C>0$.

Moreover the maps $\Phi_l$ have the property that $N(\Phi_l(\eta)) \cap \lf{\partial X}
= \emptyset$ and the local reference configurations $(\Omega^\mathrm{local}, \Phi^\eta, u^\eta)_{\eta \in \Omega_l \cap \cfcc}$ with
$\Omega^\mathrm{local} = \bigcup_{0 \in \tau \in \mathcal U} \conv(\tau)$ of
Proposition~\ref{prop:locref} can be chosen so that they are compatible, i.e.
\begin{align} \label{compfam}
\Phi^{\eta'}(\eta-\eta')=\Phi^{\eta''}(\eta-\eta'') \text{ if } \eta -\eta', \; \eta-\eta'' \in \Omega^\mathrm{local} \cap \cfcc.
\end{align}
The existence of the reference configuration $(\Omega_l,\Phi_l, u_l)$ in the case $l=1$ is a consequence of Proposition~\ref{prop:locref}. Estimate (\ref{eq:nablasoind}) follows from Lemma~\ref{lemma:wcell}.
Proposition~\ref{prop:maxnhd}
together with the assumption $\dist(\{y(x)\},y(\partial X)) \geq 2r +3$ implies that $\{\Phi_l(\eta)\}\cup N(\Phi_l(\eta))
\cap \partial X = \emptyset$ for all $\eta \in \Omega_l$. The compatibility is a consequence
of Proposition~\ref{prop:discretelocal} and (\ref{eq:ltwo}).

In the induction step we define for each $\eta \in \Omega_{l+1}
\cap \cfcc$ the label $\Phi_{l+1}(\eta)$ as follows:
$$ \Phi_{l+1}(\eta) = \left\{\begin{array}{rl}
\hat \Phi_l(\eta) & \text{ if } \eta \in \hat \Omega_l,\\
\Phi^{\eta'}(\eta-\eta') & \text{ if } \eta \in \Omega_{l+1}\setminus
\hat \Omega_l \cap \cfcc, \eta' \in \Omega_l\cap \cfcc \text{ and } |\eta-\eta'|=1,
\end{array}\right.$$
where the potentially translated domain $\Omega_l$ and map $\Phi_l$ are given by
$$\hat \Omega_l = \left\{ \begin{array}{rl} \Omega_l & \text{ if } [(l+1)/2]=[l/2]\\
\Omega_l + \eta_\mathrm{center} & \text{ else}, \end{array} \right.$$
and
$$\hat \Phi_l = \left\{ \begin{array}{rl} \Phi_l & \text{ if } [(l+1)/2]=[l/2]\\
\Phi_l(\cdot - \eta_\mathrm{center}) &\text{ else}. \end{array} \right.$$
The translated local reference configurations $(\Omega^\mathrm{local}, \Phi^\eta, u^\eta)$
are defined in a similar fashion.
We have to show that $\Phi^{\eta'}\left(\eta-\eta'\right)$ does not depend on the
choice of $\eta'$. Indeed, if $\eta',\eta'' \in \Omega_l \cap \cfcc$ have the
property that $|\eta-\eta'|=|\eta-\eta''|=1$, then $|\eta'-\eta''|=1$ since
$\Omega_l$ is a scaled octahedron with the property that $\partial \Omega_l \cap \cfcc$ is a union
of subsets of rigidly translated and rotated triangular lattices.
Thus $\Phi^{\eta'}$ and $\Phi^{\eta''}$ are
compatible, this implies that $\Phi^{\eta'}(\eta-\eta') = \Phi^{\eta''}(\eta-\eta'')$.

The existence and compatibility of the local reference configurations follows from a similar
argument like in the case $l=1$.

Inequality (\ref{eq:nablaso}) follows from (\ref{eq:ltwo}).
\subsection*{Proof of Lemma \ref{lemma:wcell}}
First, we define for each simplex $\sigma \in \mathcal D$ such that
$\sigma \subset \tau$ the local gradient
$F_\sigma = \nabla u|_\conv(\sigma)$. Note that $F_\sigma$ depends linearly on $u$ and satisfies for each $G \in \R^{3 \times 3}$ the equation $F_\sigma(u) = G$
if $u(\eta) = G\eta$ for all $\eta \in \tau$. This implies that there exists
a constant $C>0$ such that for fixed $R \in SO(3)$
\begin{align} \label{Fbound}
|F_\sigma-R|^2 \leq C\;I_R(u)
\end{align}
where
\begin{align}
I_R(u) = \min \left\{\sum_{\eta \in \tau}|u(\eta) -t-R\eta|^2 \; :
\; t \in \R^3\right\}.
\end{align}
We will show below that
\begin{align} \label{Wbound}
 \min_{R \in SO(3)}I_R(u) \leq C \amend{46}{W_{\tau}}(u) \text{ for all } u
\text{ such that } \| \dist(\nabla u,SO(3))\|_{L^\infty(\conv(\tau))}\leq c
\end{align}
if $c,C>0$ are suitably chosen.
The bounds (\ref{Fbound}) and (\ref{Wbound}) deliver the claim:
\begin{align} \label{Gbound}
\min_{R \in SO(3)} \| \amend{47}{\nabla} u - R\|^2_{L^2(\conv(\tau))} =
\min_{R \in SO(3)} \sum_{\atop{\sigma \in \mathcal D}{\sigma \subset \tau}}
\meas(\conv(\sigma))\; |F_\sigma - R|^2 \leq C \; W_\tau(u).
\end{align}
The proof of (\ref{Wbound}) rests on the observation that
\amend{48}{$I_R$} and $W_\tau$ are non-negative and invariant under translations
and rotations. Thanks to the invariances and the fact that
$\amend{48}{\min_{R \in SO(3)}I_R}(u_0)=W_\tau(u_0) = 0$, with $u_0(\eta)=\eta$, $\eta \in \tau$,
it suffices to establish the bound
\begin{align} \label{symbound}
\sum \amend{49}{|v(\eta)|}^2 \leq C W_\tau(u_0+v) \text{ for all }
v \in A, |v| \leq c.
\end{align}
To see that (\ref{symbound}) holds we define the Hessian
$H=D^2 W_\tau(u_0)$.
We will show that $H$ is positive definite on the subspace
$\amend{50}{A} \subset (\R^3)^\tau$ which is defined as the orthogonal complement of the
subspace spanned by translations $z(\eta) = \amend{51}{t} \in \R^3$ for all $\eta \in \tau$
and infinitesimal rotations $z(\eta) = A \eta,$ $\eta \in \tau$,
where $t \in \R^3$ and $A$ is skew-symmetric.

If $H$ is positive definite on $A$, then it is easy to see that there exist
constants $c, C>0$ which
depend on $\|W_\tau(u_0+\cdot)\|_{C^3(A\cap B(0,c))}$ such that
(\ref{symbound}) holds.

To prove the positivity of the restriction of $H$ to $A$
we derive a more explicit representation of $H$.
The gradient of $W_\tau$ is given by
\begin{eqnarray*}
 \frac{\partial W_{\tau}(u)}{\partial u(\eta)}=2\sum_{\substack{\eta'\in \tau\\ |\eta-\eta'|=1}}\frac{\left|u(\eta)-u(\eta')\right|-1} {\left|u(\eta)-u(\eta')\right|}\;\left(u(\eta)-u(\eta')\right),
\end{eqnarray*}
Thus, $3 \times 3$-block components the Hessian matrix $H$ are of the form,
\begin{eqnarray*}
\frac{\partial^2 W_\tau(u_0)}{\partial u(\eta)\; \partial u(\eta')} =\left\{\begin{array}{ll}2\sum\limits_{\substack{\eta'' \in \tau:\\ |\eta-\eta''|=1}}(\eta-\eta'')\otimes(\eta-\eta'') & \textnormal{if} \ \eta=\eta',\\[2.5em]
-2(\eta-\eta')\otimes(\eta-\eta') & \textnormal{if } |\eta-\eta'|=1,\\[0.5em]
0 & \textnormal{else}
\end{array}\right.
\end{eqnarray*}
for all $\eta,\eta' \in \tau$.

In the case where $\tau$ is a tetrahedron the associated eigenvalues of $H$
are $0$ (multiplicity 6), 2 (multiplicity 2), 4 (multiplicity 3) and 8
(multiplicity 1), this can be verified either with an explicit, but lengthy
calculation, or a computer-algebra package.
If $\tau$ is an octahedron we obtain the eigenvalues $0$ (multiplicity 6),
2 (multiplicity 5), 4 (multiplicity 3), $6$ (multiplicity 3) and $8$
(multiplicity 1).  In particular, both Hessian matrices have a kernel of
dimension 6.  By the rotational and translational invariance of $W_{\tau},$
it follows that zero eigenmodes must correspond to the six-dimensional space
of rotations of translations, and that the Hessian matrices are positive
definite on the orthogonal complement of this space.

\subsection*{Proof of Lemma \ref{lemma:cardinality}}
{\bf Long range pairs}\\
The lower bound is an immediate consequence of the injectivity of
the map $\Phi$ associated with each $p \in \cup_{\lambda \in \Lambda}$
(Proposition~\ref{prop:refconfig}). For each $x\in X$
and $\lambda\in\Lambda,$ let
\begin{eqnarray*}
s(x,\lambda):=\#\left\{x'\in X:(x,x')\in P(\lambda)\right\}.
\end{eqnarray*}
If $p=(x,x')\in P(\lambda),$ then
\begin{eqnarray}\label{eq:sxone}
s(x,1)= \#\left\{(x,x')\in \nnb:x'\in X\right\}=m(1).
\end{eqnarray}
The injectivity of the map $\Phi$ implies that $s(x,\lambda)\leq m(\lambda)$.
Inequality (\ref{eq:sxone}) implies $s(x,\lambda)\leq\frac{m(\lambda)}{m(1)}s(x,1).$  We obtain,
\begin{eqnarray*}
 \# P(\lambda)= \sum_{x\in X}s(x,\lambda)\leq\frac{m(\lambda)}{m(1)}\sum_{x\in X}s(x,1)=\frac{m(\lambda)}{m(1)}\# \nnb.
\end{eqnarray*}
and the left-hand inequality of (\ref{eq:cardinality}) is proved.

For the upper bound, let $\lambda\in\Lambda$ and suppose there exists
$x\in X$ such that $s(x,\lambda)<m(\lambda).$ Thanks to Proposition~\ref{Mintrin}.\ref{defex}
there exists a defect $x_b\in\partial X$ such that $y(x_b)\in B(y(x),4\lambda)$ and the minimum
distance bound (\ref{eq:mindist}) implies that
\begin{eqnarray*}
\# \left(B(y(x_b),2\lambda)\cap y(X)\right)\leq C\lambda^3.
\end{eqnarray*}
Thus, the number of labels $x\in X$ such that $s(x,\lambda)<m(\lambda)$ is bounded above by $C\lambda^3\#\partial X$ and we obtain
\begin{align*}
 \# P(\lambda)=\sum_{x\in X}s(x,\lambda)&\geq m(\lambda)\left(\frac{1}{m(1)}
\sum_{x\in X}s(x,1)-C\lambda^3\#\partial X \right)\\
&=m(\lambda)\left(\frac{1}{m(1)}\# \nnb -C\lambda^3\#\partial X\right)
\end{align*}
and the right-hand inequality of (\ref{eq:cardinality}) is proved.

{\bf Short- and medium-range pairs}\\
Firstly, note that
$$ m(1) = 12,\ m(\sqrt{2}) = 6, m(\sqrt{3}) = 24.$$
The proof of (\ref{nnbound}) is immediate:
\begin{align*}
\# \nnb =& \sum_{x\in X}\#N(x) = \sum_{x\in X_{12}}\#N(x)+\sum_{x\in X\setminus X_{12}}\#N(x)
= m(1) \# X - \sum_{x \in X \setminus X_{12}} (m(1)-\#N(x)).
\end{align*}
Proposition~\ref{prop:maxnhd} implies that $1 \leq m(1)-\#N(x) \leq m(1)$ in
the last sum, therefore
(\ref{nnbound}) holds.

Inequality (\ref{eq:X2}) is the result of a simple estimate:
\begin{align} \nonumber
\#(\Xreg \setminus \Xreg^2) \leq&\sum_{x \in \Xreg} \#(N(x) \cap (X \setminus \Xreg))
=\sum_{x \in X \setminus \Xreg} \#(N(x) \cap \Xreg)\\
\leq& 12 (n-\# \Xreg).
\end{align}

Now we consider the case $\lambda = \sqrt{2}$.
For $p \in P\left(\lambda\right)$ we define
$$\lf{ a(p) = \# \left\{ x \in \Xreg^2 \; : \;  p \subset N(x)\right\}.}$$
One obtains that
\begin{align*}
\#P\left(\lambda\right)=& \sum_{x \in \Xreg^2}
\sum_{\atop{p \subset N(x)}{p \in P(\lambda)}}\frac{1}{a(p)}\\
=& \sum_{\atop{x \in \Xreg^2}{N(x) \subset \Xreg^2}} \sum_{\atop{p \subset N(x)}{p \in P(\lambda)}}\frac{1}{a(p)} +
\sum_{\atop{x \in \Xreg
^2}{N(x) \not\subset \Xreg^2}} \sum_{\atop{p \subset N(x)}{p \in P(\lambda)}}\frac{1}{a(p)}.
\end{align*}
It is easy to see that $x \in \Xreg^2$ implies $\#\{p \subset N(x)\; : \; p \in P(\lambda)\}=24$ and $N(x) \subset \Xreg^2$ implies $a(p) = 4$, hence
\begin{align} \label{eq:2mainlb}
\#P\left(\lambda\right)
= 6\,\#\Xreg^2 +\sum_{\atop{x \in \Xreg^2}{N(x) \not\subset \Xreg^2}}\left(-6+ \sum_{\atop{p \subset N(x)}{p \in P(\lambda)}}\frac{1}{a(p)}\right).
\end{align}
Since $a(p)\leq 4$ one finds that
\begin{align} \label{eq:2est}
\sum_{\atop{p \subset N(x)}{p \in P(\lambda)}}\frac{1}{a(p)}\geq 6,
\end{align}
therefore it suffices to bound the second term in (\ref{eq:2mainlb}) from above.
\begin{align} \nonumber
&\sum_{\atop{x \in \Xreg^2}{N(x) \not\subset \Xreg^2}} \left(-6+\sum_{\atop{p \subset N(x)}{p \in P(\lambda)}}\frac{1}{a(p)}\right) \leq 18
\#\left\{x \in \Xreg^2 \; : \; N(x) \not\subset \Xreg^2\right\}\\ \nonumber
\leq&4\sum_{x \in \Xtcub^2} \#\left(N(x)\setminus \Xreg^2\right)
= 18\sum_{x \in X\setminus \Xreg^2} \underbrace{\#(N(x) \cap \Xreg^2)}_{\leq 6}\\
\leq &108(n-\#\Xreg^2)= 108(n-\#\Xreg) + 108\#(\Xreg\setminus \Xreg^2)
\leq C(n-\#\Xreg). \label{eq:2mainub}
\end{align}
The final inequality is due to (\ref{eq:X2}).
Estimates (\ref{eq:2mainub}), (\ref{eq:2est}) and (\ref{eq:2mainlb}) imply (\ref{snbound}).

Now we consider the case $\lambda = \sqrt{8/3}$ which represents the shortest distance in which the non-equivalence of fcc and hcp becomes relevant. The proofs of (\ref{snbound}) and (\ref{s8bound}) are nearly identical.
For $p \in P\left(\lambda\right)$ we define
$$ \lf{a(p) = \# \left\{ x \in \Xtcub^2 \; : \; p \subset N(x)\right\}.}$$
One obtains that
\begin{align*}
\#P\left(\lambda\right)=& \sum_{x \in \Xtcub^2}
\sum_{\atop{p \subset N(x)}{p \in P(\lambda)}}\frac{1}{a(p)}\\
=& \sum_{\atop{x \in \Xtcub^2}{N(x) \subset \Xreg^2}} \sum_{\atop{p \subset N(x)}{p \in P(\lambda)}}\frac{1}{a(p)} +
\sum_{\atop{x \in \Xtcub^2}{N(x) \not\subset \Xreg^2}} \sum_{\atop{p \subset N(x)}{p \in P(\lambda)}}\frac{1}{a(p)}.
\end{align*}
If $x \in \Xreg^2$ and $N(x) \subset \Xreg^2$ then $a(p) = 3$, hence
\begin{align} \label{eq:83mainlb}
\#P\left(\lambda\right)
= 2\,\#\Xtcub^2 +\sum_{\atop{x \in \Xtcub^2}{N(x) \not\subset \Xreg^2}}\left(-2+ \sum_{\atop{p \subset N(x)}{p \in P(\lambda)}}\frac{1}{a(p)}\right).
\end{align}
Since $a(p)\leq 3$ one finds that
\begin{align} \label{eq:83est}
\sum_{\atop{p \subset N(x)}{p \in P(\lambda)}}\frac{1}{a(p)}\geq 2,
\end{align}
therefore it suffices to bound the second term in (\ref{eq:83mainlb}) from above.
\begin{align} \nonumber
&\sum_{\atop{x \in \Xtcub^2}{N(x) \not\subset \Xreg^2}} \left(-2+\sum_{\atop{p \subset N(x)}{p \in P(\lambda)}}\frac{1}{a(p)}\right) \leq 4
\#\left\{x \in \Xtcub^2 \; : \; N(x) \not\subset \Xreg^2\right\}\\ \nonumber
\leq&4\sum_{x \in \Xtcub^2} \#\left(N(x)\setminus \Xreg^2\right)
= 4\sum_{x \in X\setminus \Xreg^2} \underbrace{\#(N(x) \cap \Xtcub^2)}_{\leq 6}\\
\leq &24(n-\#\Xreg^2)= 24(n-\#\Xreg) + 24\#(\Xreg\setminus \Xreg^2)
\leq C(n-\#\Xreg). \label{eq:83mainub}
\end{align}
The final inequality is due to (\ref{eq:X2}).
Estimates (\ref{eq:83mainub}), (\ref{eq:83est}) and (\ref{eq:83mainlb}) imply (\ref{s8bound}).

Finally we consider $\lambda = \sqrt{3}$.

For $x \in \Xreg$ we define the equator
$$N_\mathrm{eq}(x)=\left\{x' \in N(x) \; : \; \Phi^{-1}(x') \cdot (b_1\times b_2) = 0\right\}.$$
It is easy to see that
$\Phi^{-1}(N_\mathrm{eq}(x))$ is a regular hexagon in the plane spanned
by the vectors $b_1$ and $b_2$. Armed with this notation one finds
\begin{align*}
\#P\left(\sqrt{3}\right)=& 2\sum_{x \in \Xcub^2}(24-\#(N(x) \cap\Xreg^2))\\
&+\sum_{x \in \Xtcub^2}\left(36 - 2\#\left(N(x)\cap \Xreg^2\right)+\#\left(N_\mathrm{eq}(x) \cap\Xreg^2\right)\right)
\end{align*}
One obtains the following estimate for the first term:
\begin{align}\label{eq:3rep}
&2\sum_{x \in \Xcub^2}(24-\#(N(x) \cap\Xreg^2))
=24 \#\Xcub^2 +
2\,\sum_{\atop{x \in \Xcub^2}{N(x) \not \subset \Xreg^2}}(12-\,\#(N(x) \cap\Xreg^2)).
\end{align}
As $\#N(x)\leq 12$ one obtains the lower bound
\begin{align} \label{eq:3lb}
 2\sum_{x \in \Xcub^2}(24-\#(N(x) \cap\Xreg^2)) \geq 24 \#\Xcub^2.
\end{align}
The inequality $\#(N(x) \cap\Xreg^2)\geq 0$ implies
\begin{align} \nonumber
&2\sum_{x \in \Xcub^2}(24-\#(N(x) \cap\Xreg^2))
\leq 24 \#\Xcub^2 +24\#\{ x\in \Xcub^2 \; : \; N(x) \not \subset \Xreg^2\}\\
\nonumber
\leq &24 \#\Xcub^2 +24\sum_{x \in X \setminus \Xreg^2} \underbrace{\#(N(x)\setminus \Xreg)}_{\leq 12}
\leq 24 \#\Xcub^2+C(n-\#\Xreg^2)\\
\label{eq:3ub}
\leq &24 \#\Xcub^2+C(n-\#\Xreg)+C\#\left(\Xreg\setminus \Xreg^2\right)\leq 24 \#\Xcub^2+C(n-\#\Xreg).
\end{align}
The final inequality is a result of (\ref{eq:X2}).

The second term in (\ref{eq:3rep}) can be estimated in a similar way:
\begin{align*}
 &\sum_{x \in \Xtcub^2}\left(36 - 2\#\left(N(x)\cap \Xreg^2\right)+\#\left(N_\mathrm{eq}(x) \cap\Xreg^2\right)\right)\\
=& 18 \# \Xtcub
+ \sum_{\atop{x \in \Xtcub^2}{N(x) \not\subset \Xreg^2}}\left(18 - 2\#\left(N(x)\cap \Xreg^2\right)+\#\left(N_\mathrm{eq}(x) \cap\Xreg^2\right)\right).
\end{align*}
The inequality $2\#\left(N(x)\cap \Xreg^2\right)-\#\left(N_\mathrm{eq}(x) \cap\Xreg^2\right)\leq 18$ implies a lower bound for the second term in (\ref{eq:3rep}):
\begin{align} \label{eq:3tlb}
 \sum_{x \in \Xtcub^2}\left(36 - 2\#\left(N(x)\cap \Xreg^2\right)+\#\left(N_\mathrm{eq}(x) \cap\Xreg^2\right)\right)\geq 18 \# \Xtcub.
\end{align}
Similary to (\ref{eq:3ub}) one obtains the upper bound
\begin{align} \nonumber
 &\sum_{x \in \Xtcub^2}\left(36 - 2\#\left(N(x)\cap \Xreg^2\right)+\#\left(N_\mathrm{eq}(x)
\cap\Xreg^2\right)\right)\\
\leq & 18 \# \Xtcub + 18\#
\left\{x \in \Xtcub \; : \; N(x) \not \subset \Xreg^2\right\}\leq 18 \# \Xtcub+C(n-\Xreg).
\label{eq:3tub}
\end{align}
Equations (\ref{eq:3lb}), (\ref{eq:3ub}), (\ref{eq:3tlb}) and (\ref{eq:3tub}) imply (\ref{s3bound}).
\subsection*{Proof of Lemma \ref{lemma:lambda}}
For $\lambda\in\Lambda,$ let $S(\lambda):=\left\{k\in\lf{\cfcc}:\left|k\right|=\lambda\right\}$ and $M(\lambda,v):=\sum_{k\in S(\lambda)}a_k(v,B)v\cdot k.$  We first demonstrate the existence of $r(\lambda)>0$ such that $\K(\lambda):=\sum_{k\in S(\lambda)}k\otimes k=r(\lambda)\I.$  To this end, let $\left\{\pi_i\right\}_{i=1}^4$ be the four triangular lattice planes which pass through the origin and, for $i\in\left\{1,\ldots,4\right\},$ let $R_i\in SO(3)$ be the rotation by an angle $2\pi/3$ in the plane $\pi_i.$  Then, $R_iS(\lambda)=S(\lambda),$ and there exists a unique invariant line $\ell_i\subset\re^3$ such that $R_i\ell_i=\ell_i.$  Thus, $\K(\lambda)$ is invariant under $R_i$ in the sense that $R_i^T\K(\lambda)R_i=\K(\lambda),$ since
\begin{eqnarray*}
 R_i^T\K(\lambda)R_i=\sum_{k\in S(\lambda)}R_i^Tk\otimes R_i^Tk=\K(\lambda).
\end{eqnarray*}
Since $R_i^T\K(\lambda)R_i\ell_i=R_i^T\K(\lambda)\ell_i=\K(\lambda)\ell_i$ if and only if $\K(\lambda)\ell_i=\ell_i,$ this implies that $\K(\lambda)$ has four invariant lines and therefore $\K(\lambda)=r(\lambda)\I.$  In particular, if $\left\{e_1,e_2,e_3\right\}$ is the standard basis of $\re^3,$ then the relation $e_i\cdot\K(\lambda)e_i=\sum_{k\in S(\lambda)}(k\cdot e_i)^2$ for $i=1,2,3$ implies
\begin{eqnarray*}
 r(\lambda)=\frac{1}{3}\sum_{i=1}^3\sum_{k\in S(\lambda)}(k\cdot e_i)^2=\frac{1}{3}m(\lambda)\lambda^2.
\end{eqnarray*}
Thus, if $B=(v_1,v_2,v_3)\in\b$ then (\ref{eq:kbasis}) states that $a_k(v,B)=B^{-1}k\cdot B^{-1}v$ and we obtain
\begin{eqnarray*}
 M(\lambda,v)=\sum_{k\in S(\lambda)}(B^{-1}k\cdot B^{-1}v)(k\cdot v)=v\cdot B^{-T} B^{-1}\sum_{k\in S(\lambda)}(k\otimes k)v=r(\lambda)\left| B^{-1}v\right|^2=\frac{1}{3}m(\lambda)\lambda^2
\end{eqnarray*}
where the final equality follows from the fact that $v=v_i$
for some $i\in\{1,2,3\}$ and thus $|B^{-1}v|=|e_i|=1.$
\begin{notation}
\begin{enumerate}
\item \lf{$\cfcc$ and $\chcp$ denote the face-centered cubic and hexagonal close-packed lattices respectively.}
\item $B(\eta,r)\subset\re^3$ denotes the closed ball, centered at $\eta\in\re^3$ with radius $r>0.$ \item $\sx$ is the unit sphere, centered at the origin.
\item $X$ is a labeling set of $\#X$ particles, \lf{with $n=\# X.$}
\item $P:=\left\{(x,x')\in X^2:x\neq x'\right\}$ and $\nnb:=\{(x,x')\in X^2:||y(x')-y(x)|-1|\leq\alpha$ are the set of pairs and edges respectively.  We denote the components of $p\in P$ by $p=(p_1,p_2)$ \lf{and the components of $\nnb$ by $q=(q_1,q_2).$}
\item $N(x):=\left\{x'\in X:(x,x')\in \nnb\right\}$
\item $\angset(x):=\left\{q\in \nnb:q_1,q_2\in N_1(x)\right\}$ is the set of nearest neighborhood edges of $x\in X.$
\item $\cub:=\cfcc\cap \sx$ and $\tcub:=\chcp\cap \sx$ contain the vertices of a cuboctahedron and twisted cuboctahedron respectively, centered at the origin. $\oct$ denotes the octahedron with the
vertices $$
\frac{1}{\sqrt{2}}\textstyle \left(\begin{array}{cccccc}
0 & 1 & 1 & 1 & 1 & 2\\
0 & 1 & -1 & 0 & 0 & 0\\
0 & 0 & 0 & 1 & -1 & 0
\end{array}\right)\, e_i, \quad i = 1\ldots 6.$$
\item $x\in X$ is regular if $\# N(x)=12$ and $\half\#\angset(x)=24.$
\item $\partial X\subset X$ is the set of defects (c.f. Definition~\ref{defn:regular}).
\item $\d$ and $\u$ are the sets of simplices and units respectively (c.f. Definition~\ref{defn:simplices}).
\item $\nnb(\O):=\left\{q\in \nnb: \Phi^{-1}(q_1),\Phi^{-1}(q_2)\in\O\right\}\subset \nnb.$
\item $\Lambda:=\left\{|z|\;:\; z'\in\cfcc\setminus\{0\}\right\}$ is the set of fcc lattice distances and, for each $\lambda\in\Lambda,$ \  $m(\lambda)=\#\left\{\zeta\in\cfcc:\left|\zeta\right|=\lambda\right\}.$
\item For each $\lambda\in\Lambda,$ \ $P(\lambda)\subset P$ is the set of pairs associated with a reference configuration pair of length $\lambda$ (c.f. Definition~\ref{defn:kbonds}) and $P_0:=P\backslash\cup_{\lambda\in\Lambda}P(\lambda)\subset P$ is the set of defect pairs.
\item $\glab[B]$ is the set of reference paths with directions determined by a basis $B\in\b$ and $\glab=\cup_{B\in\b}\glab[B]$ is the complete set of reference paths (c.f. Definition~\ref{defn:refpaths}).
\item  $\gl(\lambda)$ is the set of label paths associated with a lattice distance
$\lambda\in\Lambda$ and $\hat \Gamma= \cup_{\lambda\in\Lambda}\hat \Gamma(\lambda)$ is the complete set of label paths (c.f. Definition~\ref{defn:refpaths}).
\item $g_{q}(\gamma)\in\{0,1\}$ is an indicator function which takes the value 1 if and only if
$q\in\gamma.$

\end{enumerate}
\end{notation}

\end{document}